\documentclass[12pt,oneside]{article}
\usepackage{times,amsmath,amsbsy,amssymb,amscd,mathrsfs,graphicx}
\usepackage{amssymb}
\usepackage{amsmath,bm}
\usepackage{amsmath,amsthm}

\usepackage{subfigure}

\usepackage{framed}
\usepackage{xcolor,color}
\usepackage[font=small,labelfont=md,textfont=it]{caption}
\usepackage{floatrow,float}
\usepackage[titletoc, title]{appendix}
\usepackage[colorlinks,linkcolor=blue,citecolor=blue]{hyperref}
\usepackage{longtable}
\usepackage{pgfplots}
\usepackage{diagbox}
\usepackage{booktabs,makecell,multirow}
\usepackage[capitalise, nosort]{cleveref}
\usepackage{algorithm}
\usepackage{algorithmic}
\usepackage{amsmath,bm}
\usepackage{cases}
\usepackage{geometry}
\usepackage{extarrows}
\usepackage{tcolorbox}
\usepackage{bbm}
\usepackage{syntonly}
\usepackage[figuresright]{rotating}
 \usepackage{epstopdf}

\crefname{equation}{}{}
\crefname{lem}{Lemma}{Lemmas}
\crefname{thm}{Theorem}{Theorems}

\newcommand{\dd}{\,{\rm d}}
\newcommand{\rF}{\rho_{_{\!F^*}}}
\newcommand{\gG}{\gamma_{_{\!G}}}
\newcommand{\R}{\,{\mathbb R}}
\newcommand{\diag}[1]{{\rm diag}\left({#1} \right)}

\newcommand{\dual}[1]{\left\langle {#1} \right\rangle}
\newcommand{\prox}[0]{ {\bf prox}}

\newcommand{\nm}[1]{\left\lVert {#1} \right\rVert}
\newcommand{\snm}[1]{\left\lvert {#1} \right\rvert}

\newcommand{\ssnm}[1]
{
	\left\vert\kern-0.25ex
	\left\vert\kern-0.25ex
	\left\vert
	{#1}
	\right\vert\kern-0.25ex
	\right\vert\kern-0.25ex
	\right\vert
}

\makeatletter
\def\spher@harm#1{%
	\vbox{\hbox{%
			\offinterlineskip
			\valign{&\hb@xt@2\p@{\hss$##$\hss}\vskip.2ex\cr#1\crcr}%
		}\vskip-.36ex}%
}
\def\gshone{\spher@harm{.}}
\def\gshtwo{\spher@harm{.&.}}
\def\gshthree{\spher@harm{.&.&.}}
\let\gsh\spher@harm
\makeatother

\newtheorem{lem}{Lemma}[section]

\newtheorem{rem}{Remark}[section]

\newtheorem{thm}{Theorem}[section]

\newcolumntype{I}{!{\vrule width 1,5pt}}
\newlength\savedwidth

\newlength\savewidth

\newcounter{mnote}
\let\oldmarginpar\marginpar
\renewcommand\marginpar[1]
{\-\oldmarginpar[\raggedleft\footnotesize #1]{\raggedright\footnotesize #1}}

\makeatletter\def\@captype{table}\makeatother

\newfloatcommand{capbtabbox}{table}[][\FBwidth]
\usepackage[T1]{fontenc} 
\usepackage[utf8]{inputenc} 
\usepackage{authblk}
\usepackage[all]{xy}
\begin{document}
\title{
  \Large \bf A continuous perspective on the inertial corrected primal-dual proximal splitting\thanks{This work was supported by the Foundation of Chongqing Normal University (Grant No. 202210000161) and the Science and Technology Research Program of Chongqing Municipal Education Commission (Grant No.
  	KJZD-K202300505).}}

\author[,1,2]{Hao Luo\thanks{Email: luohao@cqnu.edu.cn}}
\affil[1]{National Center for Applied Mathematics in Chongqing, Chongqing Normal University, Chongqing, 401331, China}
\affil[2]{Chongqing Research Institute of Big Data, Peking University,  Chongqing, 401121, China}


\date{\today}
\maketitle

\begin{abstract}
	We give a continuous perspective on the Inertial Corrected Primal-Dual Proximal Splitting (IC-PDPS) proposed by Valkonen ({\it SIAM J. Optim.}, 30(2): 1391--1420, 2020) for solving saddle-point problems. The algorithm possesses nonergodic convergence rate and admits a tight preconditioned proximal point formulation which involves both inertia and additional correction.
	Based on new understandings on the relation between the discrete step size and rescaling effect, we rebuild IC-PDPS as a semi-implicit Euler scheme with respect to its iterative sequences and integrated parameters. This leads to two novel second-order ordinary differential equation (ODE) models that are equivalent under proper time transformation, and also provides an alternative interpretation from the continuous point of view. Besides, we present the convergence analysis of the Lagrangian gap along the continuous trajectory by using proper Lyapunov functions.  
\end{abstract}

%


\section{Introduction}
\label{sec:intro-H-APDHG}
Let $X$ and $Y$ be two Hilbert spaces and $K:X\to  Y $ a bounded linear operator. We are interested in the composite optimization problem
\begin{equation}\label{eq:min-P}
	\min_{x\in X} \, \mathcal P(x):=G(x)+F(Kx) , 
\end{equation}
where $G:X\to (-\infty,+\infty]$ and $F: Y \to (-\infty,+\infty]$ are proper and lower semicontinuous convex functions. Denote by $F^*$ the Fenchel conjugate of $F$ and  reformulate \cref{eq:min-P} as a convex-concave saddle-point problem
\begin{equation}\label{eq:minmax}
	\min_{x\in X}\max_{y\in Y } \, \mathcal L(x,y) := G(x)+ \dual{Kx,y}-F^*(y).
\end{equation}
The dual problem reads equivalently as 
\begin{equation}\label{eq:min-D}
	\min_{y\in  Y } \, \mathcal D(y):= G^*(-K^* y)+F^*(y),
\end{equation}
with $K^*: Y \to X $ being the conjugate operator of $K$.

Introduce a set-valued monotone operator $H: U \rightrightarrows  U $ by that
\begin{equation}\label{eq:H}
	H(u): = \begin{pmatrix}
		\partial G(x) + K^*y\\\partial F^*(y) - Kx
	\end{pmatrix}\quad \forall\,u = (x,y)\in  U:=X\times Y.
\end{equation} 
Clearly, $\widehat{u}=(\widehat{x},\widehat{y})\in H^{-1}(0)$ if and only if $\widehat{u}$ is a saddle point of \cref{eq:minmax}, which means 
\[
\mathcal L(\widehat{x},y)\leq \mathcal L(\widehat{x},\widehat{y})\leq \mathcal L(x,\widehat{y})\quad\forall\,(x,y)\in  U .
\]
Moreover, $\widehat{x}$ and $\widehat{y}$ are respectively the optimal solutions to the primal problem \cref{eq:min-P} and the dual problem \cref{eq:min-D}.
\subsection{Literature review}
Recent years, continuous-time approach \cite{attouch_fast_2018,Luo2022d,su_dierential_2016,wibisono_variational_2016} has been extensively used to study accelerated first-order methods for unconstrained optimization, such as  Nesterov accelerated gradient (NAG) \cite{Nesterov1983} and  fast iterative shrinkage-thresholding algorithm (FISTA) \cite{beck_fast_2009}:
\begin{equation}\label{eq:nag-intro}
	\left\{
	\begin{aligned}
		{}&x^{i+1} =  \bar x^i-\tau\nabla G(\bar x^i) , \quad \tau\in(0,1/L],\\
		{}&\bar{x}^{i+1}  =x^{i+1}+\lambda_{i+1}\left(\lambda_i^{-1}-1\right)\left(x^{i+1}-x^i\right),\\
		{}&	\lambda_{i+1}^{-1}=1 / 2+\sqrt{\lambda_i^{-2}+1 / 4}.
	\end{aligned}
	\right.
\end{equation}
From continuous point of view, the use of {\it Inertia} or {\it Momentum} $\bar x^i$  usually leads to some continuous second-order ordinary differential equation (ODE): $ \mathcal F(t,x',x'',\nabla f(x)) = 0$, which gives new perspective on the acceleration mechanism. On the other hand, the numerical discretization template of the continuous model also provides an alternative way for designing and analyzing first-order algorithms that differs from classical majoriazation approach or estimate sequence \cite{Nesterov2018}. Bridging first-order and second/high-order methods (Newton or tensor), some generalized models \cite{Attouch2020f,chen_first_2019,Lin2021} involve second-order information (in space) called {\it Hessian-driving damping}: $\mathcal F(t,x',x'',\nabla f(x),\nabla^2f(x)) = 0$. This is also related to high-resolution ODEs based on backward analysis \cite{Lu2022} and high-order Taylor expansion \cite{shi_understanding_2021}.

From unconstrained minimization to the saddle-point problem \cref{eq:minmax}, however, the connection between discrete algorithms and continuous models deserves further investigations. Let us recall the well-known primal-dual hybrid gradient (PDHG) \cite{esser_general_2010} and Chambolle-Pock (CP) method \cite{chambolle_first-order_2011}. He and Yuan \cite{He2014a,he_convergence_2012} recast PDHG and CP as a compact proximal point template
\begin{equation}\label{eq:CP-PPA}
	0\in W_{i+1}H(u^{i+1}) + M_{i+1}(u^{i+1}-u^i),\quad u^i = (x^i,y^i),
\end{equation}
where the step size operator $W_{i+1}$ and the preconditioner $M_{i+1}$ are 
\[
W_{i+1}  =   \diag{\tau_iI,\sigma_{i+1}I}
,\quad M_{i+1} = \begin{pmatrix}
	I&-\tau_i K^*\\-\omega_i\sigma_{i+1} K & I
\end{pmatrix},\quad \omega_i\in[0,1].
\]
For $\omega_i=0$, it recovers PDHG, which however, is not convergent even for convex objectives \cite{He2022}. Fortunately, this instability has been overcome by CP, which actually takes $\omega_i\in(0,1]$. By choosing proper step size, the preconditioner $M_{i+1}$ keeps positive semidefinite and CP converges with the (ergodic) sublinear rate $O(1/N)$ for convex-concave problems. From the continuous perspective, Li-Shi \cite{Lic} and Luo \cite{luo_primal-dual_2022} observed that the choice $\omega_i>0$ brings subtle {\it Velocity correction} that makes the iteration/trajectory more stable; see the equilibrium analysis in \cite[Section 2.3]{luo_primal-dual_2022}. This is related to the derivative feedback in \cite{zeng_distributed_2018} and arises also from other continuous primal-dual models \cite{attouch_fast_2021,Chen2023,Chen2023a,He2024,He2021c,He2022b,Luo2021c,Luo2024b,Luo2023c,Zeng2022}.

To improve the ergodic convergence, Valkonen \cite{valkonen_inertial_2020} presented the inertial corrected primal-dual proximal splitting (IC-PDPS) based on a general corrected inertial framework (cf.\cref{eq:CP-PPA-mod-IC})
\begin{equation}\label{eq:PPI-intro}
	\left\{
	\begin{aligned}
		{}&		0\in W_{i+1}H(u^{i+1}) + M_{i+1}(u^{i+1}-\bar u^i)+\check{M}_{i+1}(u^{i+1}-u^i),\\
		{}&		\bar{u}^{i+1} = u^{i+1} + \Lambda_{i+2}(\Lambda_{i+1}^{-1}-I)(u^{i+1}-u^i),
	\end{aligned}
	\right.
\end{equation}
which combines \cref{eq:CP-PPA} with {\it Inertia} and {\it Correction} against the asymmetry part of $H$.  Equipped with the techniques of testing and gap unrolling argument, the proposed framework is highly self-contained for algorithm designing and nonergodic convergence analysis. However, to the best of our knowledge, little has been known about the continuous analogue to  \cref{eq:PPI-intro}. 
\subsection{Motivation}
Following up the continuous-time approach for unconstrained acceleration methods, a natural question is {\it can we provide an alternative perspective of the enhanced primal-dual solver IC-PDPS from the continuous point of view?} Compared with \cref{eq:CP-PPA}, IC-PDPS involves a much more complicated  combination of the primal-dual sequence $u^i $, the inertial sequence $\bar u_i$, the correction step and six parameters $\{\psi_i,\phi_i,\tau_i,\lambda_i,\mu_i,\sigma_i\}$; see \cref{eq:para-set-ic-pdps}. This makes it nontrivial to derive the continuous model. We emphasis the following two main challenges:
\begin{itemize}
	\item {\it Time step size}:  The standard way to derive the ODE model of \cref{eq:CP-PPA} is introducing some continuous ansatz $u(t_i)=u^i$ that matches the discrete iteration $u^i$ at time $t_i$. Under the constant choice $\tau_{i}=\sigma_{i+1}=\tau\in(0,1/\nm{K}]$, we can simply take the time step size $\Delta_i:=t_{i+1}-t_i=\tau$; see Li-Shi \cite{Lic} and Luo \cite{luo_primal-dual_2022}. However, this naive idea does not work for IC-PDPS since the involved parameters are more than two and it is not easy to determine such a constant choice.
	\item {\it Finite difference}: When using Taylor expansion, one has to decompose the iterative algorithm as either first-order or second-order finite difference template (cf.\cite{Lu2022,shi_understanding_2021,su_dierential_2016}). But for IC-PDPS, the primal and dual variables are coupled with each other and the inertial variable contains dynamically changing parameters. Therefore, more effort shall be paid to find the right difference formulation.
\end{itemize}

In this work, by introducing the concept of rescaled step size, we derive the continuous ODE of IC-PDPS based on proper equivalent first-order finite difference template (instead of second-order form). The IC-PDPS ODE shares the same second-order in time feature as many accelerated models and also contains symmetric velocity correction. For a detailed statement of our main contribution, see \cref{sec:main-contrib}.

\subsection{Notation}
We shall use the same bracket $\dual{\cdot,\cdot}$ as the inner products on $X$ and $ Y $, whenever no confusion arises. Denote by $I$ the identity operator and $\mathcal L(X; Y )$ the set of all bounded linear operations from $X$ to $ Y $. For $S,T\in\mathcal L(X;X)$, we say $S\geq T$ if $S-T$ is positive semi-definite. Given $x,z\in X$, the notation $\dual{x,z}_T :=\dual{Tx,z} $ makes sense for any $T\in\mathcal L(X;X)$, and it induces a semi-norm $\nm{x}_T: =\sqrt{\dual{x,x}_T}$ whenever $T$ is positive semi-definite. 

Let $\widehat{u}=(\widehat{x},\widehat{y})\in H^{-1}(0)$ be a saddle point of \cref{eq:minmax}. Throughout, we use the following two shifted convex functions
\[
\begin{aligned}
	G(x;\widehat{x}): = {}&G(x) -\frac{ \gG}{2}\nm{x-\widehat{x}}^2\quad\forall\,x\in X ,\\
	F^*(y;\widehat{y}): = {}&F^*(y) -\frac{ \rF}{2}\nm{y-\widehat{y}}^2\quad\forall\,y\in  Y ,
\end{aligned}
\]
where $\gG\geq 0$ and $\rF\geq 0$ are respectively the convexity constants of $G$ and $F^*$.
We also introduce an alternate Lagrangian
\begin{equation}\label{eq:alter-L}
	\widehat{	\mathcal L} (x,y): = G(x;\widehat{x})+ \dual{Kx,y}-F^*(y;\widehat{y})\quad\forall\,(x,y)\in  U .
\end{equation}
It is clear that both $\mathcal L(\cdot,\cdot)$ and $\widehat{\mathcal L}(\cdot,\cdot) $ have the same saddle-point(s) $\widehat{u}$. For later use, define two more functions
\[
\begin{aligned}
	\bar G(x;\widehat{x}): = {}&G(x;\widehat{x}) + \dual{K^*\widehat{y},x-\widehat{x}},\quad\bar F^*(y;\widehat{y}): = {}F^*(y;\widehat{y}) - \dual{K\widehat{x},y-\widehat{y}}.
\end{aligned}
\]
Notice that both $\bar G(\cdot;\widehat{x})$ and $\bar F^*(\cdot;\widehat{y})$ are convex. Besides, since $\widehat{u} = (\widehat{x},\widehat{y})\in H^{-1}(0)$, it holds that $\partial \bar G(\widehat{x};\widehat{x}) \ni 0$ and $\partial \bar F^*(\widehat{y};\widehat{y}) \ni 0$, which implies  $\bar G(x;\widehat{x})\geq \bar G(\widehat{x};\widehat{x})=G(\widehat{x}) $ and $\bar F^*(y;\widehat{y})\geq \bar F^*(\widehat{y};\widehat{y}) =F^*(\widehat{y})$ for all $(x,y)\in U $. Moreover, a standard calculation gives
\[
\begin{aligned}
	\mathcal L (x ,\widehat{y}) - \mathcal L (\widehat{x},y )=	{}&	\bar G(x ;\widehat{x})- G(\widehat{x}) + 
	\bar F^*(y ;\widehat{y})- F^*(\widehat{y}).
\end{aligned}
\]
\subsection{Main contribution} 
\label{sec:main-contrib}
\subsubsection{Intrinsic step size {\it v.s.} rescaled step size}
Motivated by the time rescaling effect in continuous level, we give some new understanding on the time step size of accelerated first-order methods. More precisely, in \cref{sec:NAG}, we revisit the continuous model of NAG \cref{eq:nag-intro} together with its dynamically changing parameters. By introducing the {\it intrinsic step size} and {\it rescaled step size}, we obtain two different ODEs (cf. \cref{eq:NAG-int-ode,eq:NAG-res-ode}) that are equivalent in the sense of proper time transformation. As usual, the intrinsic step size of NAG is $\Delta_i = \sqrt{\tau}$, which is related to the (proximal) gradient step length $\tau=1/L$. The rescaled step size $\Delta_i=\lambda_i$ coming from the momentum coefficient seems a little bit ``unrealistic'', which can be viewed as time rescaling in the discrete level. 

In a word, to our understanding, the intrinsic step size shall be explicitly determined by the problem parameters such as proximal step length and Lipschitz/convexity constant. The rescaled step size is more likely depending on the dynamically changing parameters. This shapes our basic idea for deriving continuous ODEs of acceleration methods that involve complicated sequences.
\subsubsection{The continuous model IC-PDPS}
For the six parameters $\{\psi_i,\phi_i,\tau_i,\lambda_i,\mu_i,\sigma_i\}$ in \cref{eq:para-set-ic-pdps}, we introduce three integrated sequences $\{\Phi_i,\Psi_i,\Theta_i\}$ in \cref{eq:aggreg-icpdps} and give the finite difference formulation \cref{eq:para-icpdps} with respect to the rescaled step size $\Delta_i = \lambda_i$. We then rewrite IC-PDPS (cf.\cref{algo:IC-PDPS} ) as a semi-implicit Euler scheme (cf.\cref{eq:diff-x-y,eq:diff-zeta-eta}). By using Taylor expansion and some technical estimates on $\{\lambda_i\}$, we derive the following continuous model of IC-PDPS:
\begin{equation}
	\label{eq:2ndODE-ICPDPS-intro}
	\frac{ 	\Upsilon(t)}{\theta(t)} u''(t)+\left[\Gamma+\frac{ 	\Upsilon(t)}{\theta(t)} \right]u'(t)  + H(u(t)) = 0,
\end{equation}
where 
\[
u(t) = (x(t),y(t)),\quad 
\Upsilon(t) = \diag{ 
	\phi(t) I, \psi(t)I },\quad\text{and}\quad  
\Gamma = \begin{pmatrix}
	\gG I&K^*\\ -K& \rF I
\end{pmatrix}.
\]
The positive parameters $\phi,\,\psi$ and $\theta$ are governed by that
\begin{equation}
	\label{eq:ode-para-intro}
	\begin{aligned}
		{} \phi'(t)=2\gG\theta(t),\quad
		{} \psi'(t)=2\rF\theta(t),\quad
		{} \theta'(t)=  \theta(t),
	\end{aligned}
\end{equation}
with arbitrary nonnegative initial conditions. Recall that the convexity constants of $G$ and $F^*$ are $\gG\geq 0$ and $\rF\geq 0$, respectively.  We call \cref{eq:2ndODE-ICPDPS-intro} the {\it rescaled IC-PDPS ODE}, which contains not only the second-order in time feature of unconstrained models \cite{Luo2022d,su_dierential_2016,shi_understanding_2021} but also the velocity correction $\Gamma u'(t)$ as that in Li-Shi \cite{Lic} and Luo \cite{luo_primal-dual_2022}. Indeed, this can be discovered more clearly from the component wise representation of \cref{eq:2ndODE-ICPDPS-intro}:
\[
\left\{
\begin{aligned}
	{}&	\frac{ 	\phi(t)}{\theta(t)} x''(t)+\left[\gG+\frac{ 	\phi(t)}{\theta(t)} \right]x'(t)  + \nabla G(x(t))+K^*(y(t)+y'(t))= 0,\\
	{}&	\frac{ 	\psi(t)}{\theta(t)} y''(t)+\left[\rF+\frac{ 	\psi(t)}{\theta(t)} \right]y'(t)  + \nabla F^*(y(t))-K(x(t)+x'(t))= 0.		
\end{aligned}
\right.
\]
%
%
%

We also equip the continuous dynamics \cref{eq:2ndODE-ICPDPS-intro} with a proper Lyapunov function 
\[
\mathcal E(t): = \theta(t)\big[\mathcal{L} (x(t),\widehat{y})-\mathcal{L} (\widehat{x},y(t))\big] + \frac{1}{2}\nm{z(t)-\widehat{u}}^2_{\Upsilon(t) },
\]
which satisfies the nonincreasing property (cf.\cref{thm:conv-lyap})
\[
\frac{\dd}{\dd t}\mathcal E(t)\leq 0\quad\Longrightarrow\quad \mathcal E(t)\leq \mathcal E(0).
\]
Notice that by \cref{eq:ode-para-intro} we have $\theta(t)=\theta(0)e^t$, which implies the exponential decay rate of the Lagrangian gap $\mathcal{L} (x(t),\widehat{y})-\mathcal{L} (\widehat{x},y(t))\leq \theta(0)\mathcal E(0)e^{-t}$.

Additionally, we also find the intrinsic step size $\Delta_i=\alpha\in(0,1/\nm{K}]$, which is motivated from the restriction of the rescaled one \cref{eq:cond-lambdai}. Based on this, we derive another model (cf.\cref{eq:int-ode-ic-pdps}) that is equivalent to \cref{eq:2ndODE-ICPDPS-intro} under proper time rescaling. Numerically, we observe that both \cref{eq:2ndODE-ICPDPS-intro,eq:int-ode-ic-pdps} are identical in the phase space and for smaller $\alpha$, the discrete sequences of IC-PDPS gives better approximation to the continuous trajectories; see \cref{fig:icpdps-intode,fig:icpdps-resode}.

\subsection{Organization}
The rest of this paper is organized as follows. In \cref{sec:framework} we will give a brief review of the general corrected inertial framework proposed by \cite[Section 2]{valkonen_inertial_2020} and restate IC-PDPS in \cref{algo:IC-PDPS} with a slight change in parameter setting. In \cref{sec:NAG}, we revisit the continuous-time limit of NAG and propose our new understanding on the step size. After that, we present our main results in \cref{sec:res-ode-icpdps,sec:int-ode-icpdps}, where two continuous-time limit ODEs of IC-PDPS shall be derived and the convergence of the trajectories will be proved by using suitable Lyapunov functions. Finally, some concluding remarks are given in \cref{sec:conclu}.

\section{A General Corrected Inertial Framework}
\label{sec:framework}
In this section, we shall briefly review the {\it general corrected inertial framework} \cite[Section 2]{valkonen_inertial_2020}, regarding the algorithm designing ingredients ({\it inertia} and {\it corrector}) and convergence analysis ({\it testing}). After that, we focus on the Inertial Corrected Primal-Dual Proximal Splitting \cite[Algorithm 4.1]{valkonen_inertial_2020}, which is a concrete example of the abstract framework \cref{eq:CP-PPA-mod-IC}. 
\subsection{The corrected inertial proximal point method}
\label{sec:icppa}
Recall the proximal point method \cref{eq:CP-PPA}  
\begin{equation}
	\label{eq:CP-PPA-mod}
	\tag{PP}
	0\in W_{i+1}H_{i+1}(u^{i+1}) + M_{i+1}(u^{i+1}-u^i),
\end{equation}
where $W_{i+1}\in\mathcal L( U ; U )$ is the {\it step length operator}.
\vskip0.3cm\noindent{\bf Inertia.} As we all know, both Nesterov's accelerated gradient (NAG) method \cite{Nesterov1983} and FISTA \cite{beck_fast_2009} adopt the same template: (proximal) gradient descent + inertial step. Applying this idea to \cref{eq:CP-PPA-mod}, it requires an inertial variable $\bar u^i$ and an invertible {\it inertial operator} $\Lambda_{i+1}\in\mathcal L( U ; U )$ and leads to the inertial scheme
\begin{equation}
	\label{eq:CP-PPA-mod-inertial}
	\left\{
	\begin{aligned}
		{}&0\in W_{i+1}H_{i+1}(u^{i+1}) + M_{i+1}(u^{i+1}-\bar u^i),\\
		{}&		\bar{u}^{i+1} = u^{i+1} + \Lambda_{i+2}(\Lambda_{i+1}^{-1}-I)(u^{i+1}-u^i).
	\end{aligned}
	\right.
\end{equation}
\vskip0.1cm\noindent{\bf Corrector.} For NAG and FISTA, a key ingredient to obtain the nonergodic faster rate $\mathcal O(1/N^2)$ is the {\it gap unrolling argument} (cf.\cite[Section 3.2]{valkonen_inertial_2020}), which, however, can not be applied to \cref{eq:CP-PPA-mod-inertial}, due to some challenging (asymmetric) component $\Gamma_{i+1}\in\mathcal L( U ; U )$ of $H_{i+1}$. This motivates the general corrected inertial method
\begin{equation}
	\label{eq:CP-PPA-mod-IC}
	\tag{PP-I}
	\left\{
	\begin{aligned}
		{}&		0\in W_{i+1}H_{i+1}(u^{i+1}) + M_{i+1}(u^{i+1}-\bar u^i)+\check{M}_{i+1}(u^{i+1}-u^i),\\
		{}&		\bar{u}^{i+1} = u^{i+1} + \Lambda_{i+2}(\Lambda_{i+1}^{-1}-I)(u^{i+1}-u^i),
	\end{aligned}
	\right.
\end{equation}
where $\check{M}_{i+1} := \Gamma_{i+1}(\Lambda_{i+1}^{-1}-I)$ denotes the {\it corrector}.
\vskip0.3cm\noindent{\bf Testing.} To establish an abstract convergence rate for the conceptual algorithm \cref{eq:CP-PPA-mod-IC}, the testing idea has been introduced \cite{valkonen_testing_2020,valkonen_acceleration_2017}. In short, it involves some proper {\it testing operator} $Z_{i+1}\in\mathcal L( U ; U )$ and works with the auxiliary sequence $\{z^i:i\in\mathbb N\}$ defined by that 
\begin{equation}\label{eq:zi}
	z^0 =u^0, \quad 
	z^{i+1} = \Lambda_{i+1}^{-1}u^{i+1} - (\Lambda_{i+1}^{-1}-I)u^i\quad\text{for }i\in\mathbb N.
\end{equation}
The following result has been proved in \cite[Theorem 2.3]{valkonen_inertial_2020}. For the sake of completeness, we restate it as below, together with its original proof. 
\begin{thm}[\cite{valkonen_inertial_2020}]
	\label{thm:conv-test-PP-I}
	Let $\{u^i:i\in\mathbb N\}\subset  U $ be generated from \cref{eq:CP-PPA-mod-IC} and the auxiliary sequence $\{z^i:i\in\mathbb N\}\subset\mathcal  U$ be defined by \cref{eq:zi}, with 
	the initial state $u^0=\bar u^0\in U$, 
	the preconditioners $\{M_{i+1}:i\in\mathbb N\}\subset \mathcal L( U ; U )$ and the invertible inertial operators $\{\Lambda_{i+1}:i\in\mathbb N\}\subset \mathcal L( U ; U )$.
	If there exist the testing operators $\{Z_{i+1}:i\in\mathbb N\}\subset \mathcal L( U ; U )$ such that, for all $i\in\mathbb N$,  
	\begin{equation}\label{eq:cond-IC_PDPS}
		\left\{
		\begin{aligned}
			{}&Z_{i+1}M_{i+1} \geq 0,\\
			{}&	Z_{i+1}M_{i+1} = M_{i+1}^*Z_{i+1}^*,\\
			{}&	\Upsilon_{i+1} + 
			2\Lambda_{i+1}^* Z_{i+1} \Gamma_{i+1}\geq \Upsilon_{i+2},
		\end{aligned}
		\right.
	\end{equation}
	then 
	\begin{equation}\label{eq:conv-est}
		\frac{1}{2}\nm{z^N-\widehat{u}}^2_{\Upsilon_{N+1}}
		+\sum_{i=0}^{N-1}\left(	\mathcal V^{i+1}(\widehat{u})+\frac{1}{2}\nm{z^{i+1}-z^i}^2_{\Upsilon_{i+1}}\right)
		\leq 	\frac{1}{2}\nm{z^0-\widehat{u}}^2_{\Upsilon_{1}}\quad\forall\,N\geq 1,
	\end{equation}
	where $\Upsilon_{i+1}: = \Lambda_{i+1}^* Z_{i+1}M_{i+1} \Lambda_{i+1}$ is self-adjoint and positive semi-definite and the gap function $ \mathcal V^{i+1}(\widehat{u})$ is defined by
	\begin{equation}\label{eq:Vhatu}
		\mathcal V^{i+1}(\widehat{u}):= \left\langle W_{i+1}H_{i+1}\left(u^{i+1}\right)-\Gamma_{i+1}\left(u^{i+1}-\widehat{u}\right), z^{i+1}-\widehat{u}\right\rangle_{\Lambda_{i+1}^* Z_{i+1}}.
	\end{equation}
\end{thm}
\begin{rem}
	\label{rem:almost}
	\cref{thm:conv-test-PP-I} gives almost the desired nonergodic convergence result. It remains to find 
	\begin{itemize}
		\item the lower bound of the summation over $\mathcal V^{i+1}(\widehat{u})$ (via gap unrolling argument),
		\item and the parameter setting (together with its growth estimates).
	\end{itemize}
\end{rem}

\subsection{The inertial corrected primal-dual proximal splitting}
\label{sec:icpdps}
Based on the conceptual algorithm template \cref{eq:CP-PPA-mod-IC}, IC-PDPS \cite[Algorithm 4.1]{valkonen_inertial_2020} can be derived from the following choices (cf.\cite[Eq.(4.1)]{valkonen_inertial_2020}):
\begin{itemize}
	\item $	\textit{Inertial operator: }
	\Lambda_{i+1} := \diag{ 
		\lambda_iI,\mu_{i+1}I }$ with $\lambda_i,\,\mu_{i+1}>0$,
	\item $	\textit{Testing operator: }
	Z_{i+1} := \diag{ 
		\phi_iI,\psi_{i+1}I }$ with $\phi_i,\,\psi_{i+1}>0$,
	\item $	\textit{Step length operator: }
	W_{i+1} := \diag{ 
		\tau_iI,\sigma_{i+1}I }
	$ with $\tau_i,\,\sigma_{i+1}>0$,
	\item $\textit{Preconditioner: }
	M_{i+1} :=\begin{pmatrix}
		I & -\mu_{i+1}^{-1}\tau_i K^*\\
		-\lambda_i^{-1}\sigma_{i+1}\omega_i K& I
	\end{pmatrix}$ with $\omega_i = \frac{\lambda_{i}\phi_i\tau_i}{\lambda_{i+1}\phi_{i+1}\tau_{i+1}}$,
	\item $\textit{Corrector: }
	\check{M}_{i+1}: = \Gamma_{i+1}(\Lambda_{i+1}^{-1}-I)$ with $\Gamma_{i+1}= W_{i+1}\Gamma
	$ and $ \Gamma = \begin{pmatrix}
		\gG  I&K^*\\ -K& \rF I
	\end{pmatrix}$.
\end{itemize}
Indeed, with this setting, 
%
the condition \cref{eq:cond-IC_PDPS} in \cref{thm:conv-test-PP-I} becomes (cf.\cite[Lemmas 4.1 and 4.2]{valkonen_inertial_2020})
	\begin{equation}\label{eq:cond-ic-pdps}
		\left\{
		\begin{aligned}
			{}&\psi_{i+1}\mu_{i+1}^2\geq \phi_i\tau_i^2\nm{K}^2,\\
			{}&	\lambda_{i+1} \phi_{i+1} \tau_{i+1}=\mu_{i+1} \psi_{i+1} \sigma_{i+1},\\
			{}&	\lambda_i^2 \phi_i\left(1+2 \gG  \tau_i \lambda_i^{-1}\right)   \geq \lambda_{i+1}^2 \phi_{i+1} ,\\
			{}&	\mu_{i+1}^2 \psi_{i+1}\left(1+2 \rF \sigma_{i+1} \mu_{i+1}^{-1}\right)   \geq \mu_{i+2}^2 \psi_{i+2},
		\end{aligned}
		\right.
	\end{equation}
	which promises that (cf.\cref{eq:conv-est})
	\begin{equation}\label{eq:conv-est-icpdps}
		\frac{1}{2}\nm{z^N-\widehat{u}}^2_{\Upsilon_{N+1}}
		+\sum_{i=0}^{N-1}\left(	\mathcal V^{i+1}(\widehat{u})+\frac{1}{2}
		\nm{z^{i+1}-z^i}^2_{\Upsilon_{i+1}}\right)
		\leq 	\frac{1}{2}\nm{z^0-\widehat{u}}^2_{\Upsilon_{1}},
	\end{equation}
	where
	\[
	\Upsilon_{i+1} =  \Lambda_{i+1}^* Z_{i+1}M_{i+1} \Lambda_{i+1} = 
	\begin{pmatrix}
		\phi_i\lambda_{i}^2I & -\lambda_{i}\phi_i\tau_i K^*\\
		- \lambda_{i}\phi_i\tau_i
		K& \psi_{i+1}\mu_{i+1}^2I
	\end{pmatrix},
	\]
	and the gap function $\mathcal V^{i+1}(\widehat{u})$ (cf.\cref{eq:Vhatu}) is
	\[
	\mathcal V^{i+1}(\widehat{u})=\lambda_{i}\phi_i\tau_i\dual{\partial \bar G(x^{i+1};\widehat{x}),\zeta^{i+1}-\widehat{x}}+\mu_{i+1}\psi_{i+1}\sigma_{i+1}
	\dual{\partial \bar F^*(y^{i+1};\widehat{y}),\eta^{i+1}-\widehat{y}}.
	\]
	Above, recall that the auxiliary sequence $z^i = (\zeta^i,\eta^i)\in U$  consists of (cf.\cref{eq:zi})
	\begin{equation} \label{eq:zetai-etai}
		\left\{
		\begin{aligned}
			\zeta^0 ={}& x^0,	&&\zeta^{i+1} = \lambda_{i}^{-1}x^{i+1} - (\lambda_{i}^{-1}-1)x^i,\\
			\eta^0 ={}&y^0,&&
			\eta^{i+1} = \mu_{i+1}^{-1}y^{i+1} - (\mu_{i+1}^{-1}-1)y^i.
		\end{aligned}
		\right.
	\end{equation}

	\vskip0.1cm\noindent{\bf Gap unrolling argument.}		
	In \cite[Lemma 3.2]{valkonen_inertial_2020}, the gap unrolling argument is developed to establish the lower bound of the summation over $\mathcal V^{i+1}(\widehat{u})$, which gives the estimate of the objective (Lagrangian) gap. In short, under the recurrence condition 		
	\begin{equation}\label{eq:phitaui}
		\left\{
		\begin{aligned}
			{}&			\phi_i \tau_i  \geq(1- \lambda_{i+1}) \phi_{i+1}\tau_{i+1},\\
			{}&	\psi_{i+1}\sigma_{i+1} \geq(1- \mu_{i+2}) \psi_{i+2}\sigma_{i+2},
		\end{aligned}
		\right.
	\end{equation}
	a telescope procedure turns  \cref{eq:conv-est-icpdps} into 
	\begin{equation}\label{eq:conv-icpdps}
		\begin{aligned}
			{}&\phi_{N-1}\tau_{N-1}\left( \bar G(x^{N};\widehat{x})- G(\widehat{x})\right)
			+\psi_{N}\sigma_N\left(\bar F^*(y^{N};\widehat{y})- F^*(\widehat{y})\right)\\
			{}&\qquad+\frac{1}{2}\nm{z^N-\widehat{u}}^2_{\Upsilon_{N+1}}
			+\frac{1}{2}\sum_{i=0}^{N-1} 
			\nm{z^{i+1}-z^i}^2_{\Upsilon_{i+1}} \leq 	\frac{1}{2}\nm{z^0-\widehat{u}}^2_{\Upsilon_{1}}+C_0(\widehat{u}),
		\end{aligned}
	\end{equation}
	where $	C_0(\widehat{u}):=\psi_{1}\sigma_1(1-\mu_1)\left(\bar F^*(y^0;\widehat{y})-F^*(\widehat{y})\right)
	+\phi_{0}\tau_{0}(1-\lambda_0)\left(\bar G(x^{0};\widehat{x})- G(\widehat{x})\right)$.
	\vskip0.3cm\noindent{\bf Parameter setting and overall algorithm.} 
	In view of \cref{eq:conv-icpdps}, it remains to determine the parameter sequences $\{\psi_i,\phi_i,\tau_i,\lambda_i,\mu_i,\sigma_i:\,i\in\mathbb N\}$ and establish the growth estimates of $\phi_N\tau_N$ and $\psi_N\sigma_N$.
	Collecting \cref{eq:cond-ic-pdps,eq:phitaui} yields the final parameter setting
	\begin{equation}\label{eq:para-set-ic-pdps}
		\left\{
		\begin{aligned}
			{}&\psi_{i+1}\mu_{i+1}^2\geq \phi_{i}\tau_{i}^2\nm{K}^2,\\		
			{}&	\lambda_{i+1} \phi_{i+1} \tau_{i+1}=\mu_{i+1} \psi_{i+1} \sigma_{i+1},\\			
			{}&  \phi_i \tau_i  \geq(1- \lambda_{i+1}) \phi_{i+1}\tau_{i+1},\\
			{}&	\psi_{i+1}\sigma_{i+1} \geq(1- \mu_{i+2}) \psi_{i+2}\sigma_{i+2},\\
			{}&	\lambda_{i+1}^2 \phi_{i+1}\leq \lambda_i^2 \phi_i+2 \gG  \phi_i\tau_i \lambda_i ,\\
			{}&	\lambda_{i+1}^2 \psi_{i+1}\leq \lambda_{i}^2 \psi_{i}+2 \rF \psi_{i}\tau_{i} \lambda_{i}.
		\end{aligned}
		\right.
	\end{equation}
	Notice that \cref{eq:CP-PPA-mod-IC} itself involves only four sequences $\{\tau_i,\lambda_i,\mu_i,\sigma_i:\,i\in\mathbb N\}$ and the left two $\{\psi_i,\phi_i:\,i\in\mathbb N\}$ come from the testing operator in \cref{thm:conv-test-PP-I}.
	%
	In \cite[Theorems 4.5-4.8]{valkonen_inertial_2020}, with additional constraints (see \cite[Eq.(4.10c)]{valkonen_inertial_2020})
	\begin{equation}\label{eq:para-icpdps-1}
		\phi_{i+1}\tau_{i+1} = \psi_{i+1}\sigma_{i+1},\quad \lambda_{i+1} = \mu_{i+1},
	\end{equation}
	specific choices of the parameters have been given under four scenarios: (i) $\gG=\rF=0$, (ii) $\gG=0<\rF$, (ii) $\gG>0=\rF$ and (iv) $ \gG\rF >0$.
	
	For simplicity, we consider the equality case
	\begin{equation} \label{eq:para-icpdps-2}
		\left\{
		\begin{aligned}			
			{}&  \phi_i \tau_i  =(1- \lambda_{i+1}) \phi_{i+1}\tau_{i+1},\\
			{}&	\lambda_{i+1}^2 \phi_{i+1}-\lambda_i^2 \phi_i= 2 \gG  \phi_i\tau_i \lambda_i ,\\
			{}&	\lambda_{i+1}^2 \psi_{i+1}-\lambda_{i}^2 \psi_{i}=2 \rF \psi_{i}\tau_{i} \lambda_{i},\\
			{}&\psi_{i}\lambda_{i}^2= \phi_{i}\tau_{i}^2/\alpha^2,\quad\alpha\in(0,1/\nm{K}],
		\end{aligned}
		\right.
	\end{equation}
	which gives a simplified version of \cref{eq:para-set-ic-pdps}.
	In this last equation, we shift $\psi_{i+1}\lambda_{i+1}^2$ to $\psi_{i}\lambda_{i}^2$. This meets the first condition in \cref{eq:para-set-ic-pdps} since $\mu_{i+1}=\lambda_{i+1}$ and $\psi_{i+1}\lambda_{i+1}^2\geq\psi_{i}\lambda_{i}^2$. Finally, combining \cref{eq:CP-PPA-mod-IC} with \cref{eq:para-icpdps-1,eq:para-icpdps-2} leads to \cref{algo:IC-PDPS}. We mention that \cref{algo:IC-PDPS} differs from the original IC-PDPS \cite[Algorithm 4.1]{valkonen_inertial_2020} only in the parameter setting.
		
		\begin{algorithm}[H]
			\caption{IC-PDPS}
			\label{algo:IC-PDPS} 
			\begin{algorithmic}[1] 
				\REQUIRE  Problem parameters: $\gG ,\,\rF \geq 0$ and $\alpha\in(0,1/\nm{K}]$.			\\
				\quad ~~~Initial guesses: $\zeta^0,\,x^0\in\R^n$ and  $\eta^0,\,y^0\in\R^m$.\\
				\quad ~~~Initial parameters: $\phi_0,\psi_0,\tau_0>0$.
				\FOR{$i=0,1,\ldots$} 
				\STATE Compute $\lambda_i = \sqrt{\phi_i/\psi_i}\,\tau_i/\alpha$.
				\STATE Compute $c_i =  \lambda_{i}^2 \phi_{i}+2 \gG \phi_{i}\tau_{i} \lambda_{i}$ and $ d_i=\lambda_{i}^2 \psi_{i}+2 \rF \psi_{i}\tau_{i} \lambda_{i}$.
				\STATE Update $\tau_{i+1} = \frac{d_i\alpha^2}{\sqrt{c_id_i}\alpha+\phi_i\tau_i },\,\phi_{i+1} = \frac{d_i\alpha^2}{\tau_{i+1}^2}$ and $\psi_{i+1} =\frac{d_i^2\alpha^2}{c_i\tau_{i+1}^2}$.
				\STATE Compute $\lambda_{i+1} = \sqrt{c_i/d_i}\,\tau_{i+1}/\alpha$ and $ \sigma_{i+1} = \tau_{i+1}c_i/d_i$.
				\STATE
				\STATE Compute $s_i = \gG \tau_i(\lambda_i^{-1}-1)$  and $\widehat{x}^i=x^i+\lambda_i(\zeta^i-x^i) /\left(1+s_i\right)$.
				\STATE Update $x^{i+1}=\prox_{\widetilde{\tau}_i  G}(	\widehat{x}^i -\widetilde{\tau}_i  K^*\eta^i ),\quad \widetilde{\tau}_i = \tau_i/(1+s_i)$.\label{algo:x}			
				\STATE Update $\zeta^{i+1}=x^i+\lambda_i^{-1}(x^{i+1} -x^i)$. \label{algo:zeta}			
				\STATE
				\STATE Compute  $t_{i+1} = \rF \sigma_{i+1}(\lambda_{i+1}^{-1}-1)$ and $	\widehat{y}^i=y^i+\lambda_{i+1}(\eta^i-y^i) /\left(1+t_{i+1}\right) $.
				\STATE Compute $ 	\bar{u}^{i+1}=\zeta^{i+1}+\omega_i\left(\zeta^{i+1}-\zeta^i\right)$ with $  \omega_i = \frac{\lambda_{i}\phi_i\tau_i}{\lambda_{i+1}\phi_{i+1}\tau_{i+1}}$.
				\STATE Update $y^{i+1}=\prox_{\widetilde{\sigma}_{i+1} F^*}\left(	\widehat{y}^i + \sigma_{i+1} K\bar u^{i+1}\right),\quad \widetilde{\sigma}_{i+1}= \sigma_{i+1}/(1+t_{i+1})$.	
				\label{algo:y}				 
				\STATE Update $\eta^{i+1}=y^i+\lambda_{i+1}^{-1}(y^{i+1} -y^i)$.	\label{algo:eta}
				\ENDFOR
			\end{algorithmic}
		\end{algorithm}
\section{Continuous-time NAG Revisited}
\label{sec:NAG}
In this section, we revisit the continuous-time model of Nesterov accelerated gradient (NAG) by using our new understanding on the relation between the step size and the rescaling effect. To simplify the presentation but maintain the key idea, let us start from the standard NAG \cite{Nesterov2018} (or FISTA \cite{beck_fast_2009} without proximal calculation):
\begin{equation}\label{eq:NAG}
	\tag{NAG} 
	\left\{
	\begin{aligned}
		{}&x^{i+1} =  \bar x^i-\tau\nabla G(\bar x^i) , \quad \tau\in(0,1/L],\\
		{}&\bar{x}^{i+1}  =x^{i+1}+\lambda_{i+1}\left(\lambda_i^{-1}-1\right)\left(x^{i+1}-x^i\right),\\
		{}&	\lambda_{i+1}^{-1}=1 / 2+\sqrt{\lambda_i^{-2}+1 / 4},
	\end{aligned}
	\right.
\end{equation}
which solves \cref{eq:min-P} only with the smooth objective $G$ that has $L$-Lipschitz continuous gradient. Notice that \cref{eq:NAG} corresponds to \cref{eq:CP-PPA-mod-IC} with $H_{i+1}(x)=  \nabla G\left(\bar x^i\right) $ and the following setting:
\begin{itemize} 
	\item $	\textit{Testing operator: }
	Z_{i+1} =  
	\phi_iI $ with $\phi_{i+1}=\phi_i/(1-\lambda_{i+1})$,
	\item $	\textit{Step length operator: }
	W_{i+1} = 
	\tau I
	$ with $\tau\in(0,1/L]$,	
	\item $	\textit{Inertial operator: }
	\Lambda_{i+1} = 
	\lambda_iI$ with $\lambda_i>0$,	
	\item $\textit{Preconditioner: }
	M_{i+1} =I$,
	\item  
	$\textit{Corrector : }\check{M}_{i+1} = 0$.
\end{itemize}
The last equation in \cref{eq:NAG} for the sequence $\{\lambda_i:\,i\in\mathbb N\}$ actually meets the condition \cref{eq:cond-IC_PDPS}.

Although the low-resolution ODE and its high-resolution variant of \cref{eq:NAG} with explicit parameters $\lambda_{i} = 2/(i+2)$ have been derived by Su et al. \cite{su_dierential_2016} and Shi et al. \cite{shi_understanding_2021}, we will further discuss more about the step size and rebuild the low-resolution models of \cref{eq:NAG} together its parameter sequence $\{\lambda_i:\,i\in\mathbb N\}$  via proper first-order difference-like formulations. 
\subsection{Intrinsic NAG ODE}
Recall the auxiliary sequence $\{z^i:\,i\in\mathbb N\}$ in \cref{eq:zi}, which leads to
\begin{equation}\label{eq:barxi}
	x^{i+1}-\bar  x^i 
	={} \lambda_{i}(z^{i+1}-z^i),\quad
	z^{i+1}-x^{i+1}={} 	\left(\lambda_i^{-1}-1\right)\left(x^{i+1}-x^i\right).
\end{equation}
Based on this, we drop the inertial sequence $\{\bar x^i:\,i\in\mathbb N\}$ and rewrite \cref{eq:NAG} as follows
\begin{equation}\label{eq:diff-NAG}
	\left\{
	\begin{aligned}
		\lambda_{i}(z^{i+1}-z^i)= 			&-\tau \nabla G(\bar x^i) ,\\
		x^{i+1} -x^i=			{}&\lambda_i(	z^{i+1} -x^i),\\
		\lambda_{i+1}^{-2}-\lambda_{i}^{-2}={}&\lambda_{i+1}^{-1}.
	\end{aligned}
	\right.
\end{equation}
To derive the continuous-time ODE, the standard way is firstly introducing proper time $s_i$ and the ansatz $\widetilde{x}(s_i )=x^i$ and then plugging them into \cref{eq:diff-NAG} to find the leading term. But the problem is how to determine the ``time step size'' $\Delta_i=s_{i+1}-s_i $ along the trajectory $\widetilde{x}(s_i)$?  Note that $\tau$ is the proximal step size but not a reasonable candidate for the time step size since from the spectral analysis in \cite[Section 2]{Luo2022d}, acceleration method allows a larger one $ \sqrt{\tau}$. 

Motivated by this, a more natural one shall be $
\Delta_i=\sqrt{\tau}$, which is called the {\it intrinsic step size}. Observing that the second and the third equations in \cref{eq:diff-NAG} have no explicit relation with $\sqrt{\tau}$, we rearrange \cref{eq:diff-NAG} to obtain
\begin{equation}\label{eq:intrinsic-NAG-ode}
	\tag{Intrinsic NAG}
	\left\{
	\begin{aligned}
		{}& \frac{z^{i+1}-z^i}{\sqrt{\tau}}= - \frac{\sqrt{\tau}}{\lambda_{i}}  \nabla G(\bar x^i) , \\
		{}&	\frac{\sqrt{\tau}}{\lambda_{i}}\cdot\frac{x^{i+1}-x^i}{\sqrt{\tau}}= z^{i+1}  -x^i,\\
		{}&\frac{\left[\frac{\sqrt{\tau}}{\lambda_{i+1}}\right]^{2}-\left[\frac{\sqrt{\tau}}{\lambda_{i}}\right]^{2}}{\sqrt{\tau}}=\frac{\sqrt{\tau}}{\lambda_{i+1}}.
	\end{aligned}
	\right.
\end{equation}
The above rewriting is called the intrinsic NAG, which gives the finite difference form of \cref{eq:NAG} together with its parameter. Introduce the intrinsic time 
\begin{equation}\label{eq:time-iss}
	s_i =\sum_{j=0}^{i-1}\Delta_j=i\sqrt{\tau}
\end{equation}
and the ansatz $\widetilde{X}(s_i):= (\widetilde{x}(s_i ),\, \widetilde{{z}}(s_i ),\, \widetilde{\theta}(s_i )) = (x^i,z^i,\sqrt{\tau}/\lambda_i)$. Assuming the objective $G(x)$ is smooth enough, we can use Taylor expansion and ignore the high order term $O(\sqrt{\tau})$ to get the {\it intrinsic NAG ODE} (cf. \cref{app:NAG-int-ode})
\begin{equation}\label{eq:NAG-int-ode}
	\begin{aligned}
		{}&	 \widetilde{z}'(s) =- \widetilde{\theta}(s)\nabla G(\widetilde{x}(s)), \quad 
		{}\widetilde{\theta}(s)	\widetilde{x}'(s) = \widetilde{z}(s)-\widetilde{x}(s),\quad \widetilde{\theta}'(s)=1/2,\quad s\geq 0.
	\end{aligned}
\end{equation} 
In addition, eliminating $\widetilde{z}(s)$ gives a second-order ODE
\begin{equation}\label{eq:NAG-int-ode-2nd}
	\widetilde{x}''(s) + \frac{3}{2\widetilde{\theta}(s)}\widetilde{x}'(s) + \nabla G(\widetilde{x}(s)) = 0,\quad s\geq 0.
\end{equation}
\begin{rem}
	Actually, we have $  \widetilde{\theta}(s) = \widetilde{\theta}(0)+s/2$. Hence, \cref{eq:NAG-int-ode} is closely related to the ODE derived by Su et al. \cite{su_dierential_2016}, which corresponds to $\widetilde{\theta}(0)=0$. We also note that both Su et al. \cite{su_dierential_2016} and Shi et al. \cite{shi_understanding_2021} used the time \cref{eq:time-iss} and worked only with the ansatz $\widetilde{x}(s_i)=x^i$. However, our derivation is based on the first-order difference-like formulation \cref{eq:intrinsic-NAG-ode}, which involves both the iteration sequence $(x^i,z^i)$ and the parameter $\lambda_i$.
\end{rem}
\subsection{Rescaled NAG ODE}
In view of the roles that $\tau$  and  $\lambda_i$ play in \cref{eq:NAG}, we confirm that they are totally different. More precisely, $\tau $ is the proximal step size while $\lambda_{i}$ is used for the inertial sequence. Inspired by the time rescaling effect in continuous level \cite[Section 3.2]{Luo2022d}, we take another look at \cref{eq:NAG} by considering $	\Delta_i = \lambda_i$, which is called the {\it rescaled step size}.
The sharp decay behavior of $\lambda_i$ is given in \cref{lem:ub-lb-lambdai}, which shows if we ignore the logarithm factor, then $\lambda_i\sim 2/(i+2/\lambda_0)$ and the rescaled time becomes
\begin{equation}\label{eq:time-rss}
	t_i =\sum_{j=0}^{i-1}\Delta_j\sim  2\ln (1+\lambda_0i/2).
\end{equation}

To work with the rescaled step size $	\Delta_i = \lambda_i$, we reformulate \cref{eq:diff-NAG} by that
\begin{equation}\label{eq:rescale-NAG-ode}
	\tag{Rescaled NAG}
	\left\{
	\begin{aligned}
		{}&\frac{z^{i+1}-z^i}{\lambda_{i}}= -\frac{\tau}{\lambda_{i}^2}  \nabla G(\bar x^i) , \\
		{}&	\frac{x^{i+1}-x^i}{\lambda_{i}}=z^{i+1}  -x^i,\\
		{}&\frac{\frac{\tau}{\lambda_{i+1}^2}-\frac{\tau}{\lambda_{i}^2}}{\lambda_{i+1}}=\frac{\tau}{\lambda_{i+1}^2},
	\end{aligned}
	\right.
\end{equation}
which is called the rescaled NAG. Similarly as before, introducing the  ansatz
\[
X(t_i):=(x(t_i ),{z}(t_i ),\theta (t_i )) = (x^i,z^i,\sqrt{\tau}/\lambda_i),
\]
we are able to derive the {\it rescaled NAG ODE} (cf. \cref{app:NAG-res-ode})
\begin{equation}\label{eq:NAG-res-ode}
	\begin{aligned}
		{}&	{z}'(t) =-\theta^2(t)  \nabla G(x(t)), \quad
		{}	x'(t) = {z}(t)-x(t),\quad \theta '(t) = \theta(t)/2,\quad  t\geq 0.
	\end{aligned}
\end{equation} 
Eliminating ${z}(t)$ gives a second-order ODE
\[
x''(t) +x'(t) +\theta^2(t) \nabla G(x(t)) = 0,\quad t\geq 0.
\]
Since $\theta(t) = \theta(0)e^{t/2}$, this coincides with the Nesterov  accelerated gradient flow model (for convex case) derived in \cite[Section 3.1]{Luo2022d}.

\begin{figure}[H]
	\centering
	\includegraphics[width=14cm]{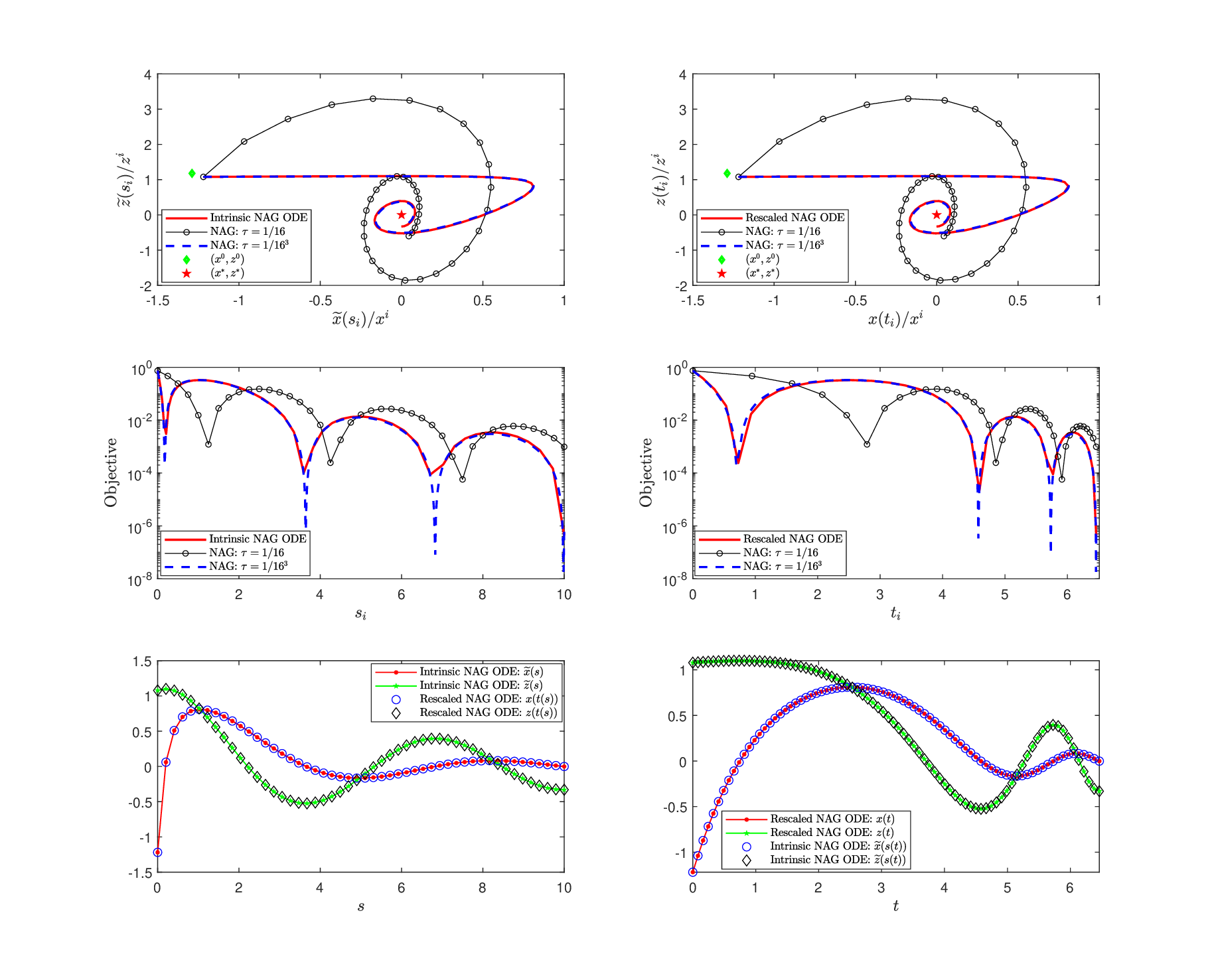}
	\caption{Numerical illustration of \cref{eq:NAG}, intrinsic NAG ODE \cref{eq:NAG-int-ode} and rescaled NAG ODE \cref{eq:NAG-res-ode} with the same initial values. The objective is simply quadratic $G(x)=x^2/2$ in one dimension.  In the top line, we see smaller $\tau$ leads to better approximations of the continuous trajectories. In the middle line, we also report the convergence behavior of the objective. In the bottom line, $X(t)$ and $\widetilde{X}(s)$ are different in terms of their own time coordinates but under the rescale relations $t(s)=2\ln(1+0.5s/\theta(0))$ and $s(t)=2\theta(0)(e^{t/2}-1)$, they coincide with each other. This also agrees with what we have observed from the phase space (top line).}
	\label{fig:NAG-ode-test1}
\end{figure}

\begin{rem}	
	In \cite[Section 3.2]{Luo2022d}, it has been discussed that \cref{eq:NAG-int-ode,eq:NAG-res-ode} are equivalent in the sense of time rescaling. Namely, under the relation $t = 2\ln (1+0.5s/\theta(0))$, we have $
	X(t)=\widetilde{X}(s)$. In view of \cref{eq:time-iss,eq:time-rss} and the relation $ \theta _0=\sqrt{\tau}/\lambda_0$, we find the equivalence of discrete time $t_i \sim 2\ln(1+0.5 s_i /\theta_0)$. The following diagram further explains the connection between different forms of \cref{eq:NAG} and the corresponding ODE models. For a numerical illustration, see \cref{fig:NAG-ode-test1}. 
	
	\[
		\xymatrix{
			& 
			\text{\cref{eq:intrinsic-NAG-ode} $\widetilde{X}(s_i)$} \ar[d] \ar[rr]^{ }  &&   \text{Intrinsic NAG ODE \cref{eq:NAG-int-ode} $\widetilde{X}(s)$}  \ar[d]         \\
			\text{\cref{eq:NAG}} \ar[dr]_{ \Delta_i=\lambda_i}  \ar[ur]^{ \Delta_i=\sqrt{\tau}} 	&{ t_i }\sim 2\ln {(1+ 0.5s_i /\theta_0)}&&{t}= 2\ln { (1+0.5s/\theta(0))}\\
			&   \text{\cref{eq:rescale-NAG-ode} $X(t_i)$}       \ar[u]
			\ar[rr]^{ } &&    \text{Rescaled NAG ODE \cref{eq:NAG-res-ode} $X(t)$}
			\ar[u]
		}     
	\] 
\end{rem} 
\begin{rem}
	Our novel understanding on the discrete time step size for \cref{eq:NAG} shapes our basic idea for deriving continuous-time model of acceleration methods that involve complicated sequences and parameters. Usually, we can work with the intrinsic step size that can be determined explicitly from the iteration (proximal step or gradient step). However, for general cases, it is not easy to find the intrinsic step size, especially for primal-dual methods. As an important application, in the next section we will further show that the idea of rescaled step size does make sense. This paves the way for establishing the continuous model of \cref{algo:IC-PDPS}  which is more complicated than \cref{eq:NAG}.
\end{rem}
\section{The Continuous Model of IC-PDPS}
\label{sec:res-ode-icpdps}
In this section, by assuming the objectives $G$ and $F^*$ are smooth enough, we shall derive the continuous ODE of \cref{algo:IC-PDPS}:
\begin{equation}\label{eq:ode-ic-pdps}
	\left\{
	\begin{aligned}
		{}&x'(t)=  \zeta(t)-x(t),\\
		{}&y'(t)=  \eta(t)-y(t),\\		
		{}&	\phi(t) \zeta'(t) = \theta(t)\left[\gG(x(t)-\zeta(t))-\nabla G(x(t))-K^*\eta(t)\right],\\
		{}&	\psi(t) \eta'(t) = \theta(t)\left[\rF(y(t)-\eta(t))-\nabla F^*(y(t))+K\zeta(t)\right],
	\end{aligned}
	\right.
\end{equation}
where $\phi,\,\psi$ and $\theta$ are governed by the continuous limit of \cref{eq:para-icpdps-2}:
\begin{equation}\label{eq:ode-para}
	{}\phi'(t)= 2\gG\theta(t),\quad
	{} \psi'(t)=2\rF\theta(t),\quad
	{} \theta'(t)= \theta(t).
\end{equation}  
Eliminating $\zeta(t)$ and $\eta(t)$ gives a second-order ODE system
\begin{equation}\label{eq:ic-pdps-2nd-ode}
	\left\{
	\begin{aligned}
		{}&	\frac{ 	\phi(t)}{\theta(t)} x''(t)+\left[\gG+\frac{ 	\phi(t)}{\theta(t)} \right]x'(t)  + \nabla G(x(t))+K^*(y(t)+y'(t))= 0,\\
		{}&	\frac{ 	\psi(t)}{\theta(t)} y''(t)+\left[\rF+\frac{ 	\psi(t)}{\theta(t)} \right]y'(t)  + \nabla F^*(y(t))-K(x(t)+x'(t))= 0.		
	\end{aligned}
	\right.
\end{equation}
Moreover, introducing $u(t)=(x(t),y(t))$ and $	\Upsilon(t) =  \diag{ 
	\phi(t)I, \psi(t)I }$, we can rearrange \cref{eq:ic-pdps-2nd-ode} as a more compact form
\begin{equation}
	\label{eq:2ndODE-ICPDPS}
	\frac{ 	\Upsilon(t)}{\theta(t)} u''(t)+\left[\begin{pmatrix}
		\gG  I&K^*\\ -K& \rF I
	\end{pmatrix}+\frac{ 	\Upsilon(t)}{\theta(t)} \right]u'(t)  + H(u(t)) = 0.
\end{equation}

As illustrated in the last section, our main idea is combining the first-order difference-like formulation with proper time step size. To do so, we firstly establish the parameter equation \cref{eq:ode-para} by introducing some new integrated sequences in \cref{sec:agg-seq}. After that, in \cref{sec:icpdps-ode} we rewrite \cref{algo:IC-PDPS} carefully into a difference template that leads to the desired model \cref{eq:ode-ic-pdps} and present a Lyapunov analysis in \cref{eq:lyap-res-icpdps-ode}. Following the spirit of the last section, we call \cref{eq:ode-ic-pdps}/\cref{eq:ic-pdps-2nd-ode}/\cref{eq:2ndODE-ICPDPS} the {\it rescaled IC-PDPS ODE} since we use the rescaled time \cref{eq:ti}. 
\subsection{Continuous parameter equation}
\label{sec:agg-seq}
%
Define the following integrated sequences
\begin{equation}\label{eq:aggreg-icpdps}
	\Phi_i = \phi_i\lambda_{i}^2,\quad \Psi_i = \psi_i\lambda_i^2,\quad \Theta_i = \phi_i\tau_i.
\end{equation}
This leads to an equivalent presentation of \cref{eq:para-icpdps-2}:
\begin{equation}
	\label{eq:para-icpdps}
	\frac{\Phi_{i+1}-\Phi_i}{\lambda_i} = 2\gG \Theta_i,\quad
	\frac{\Psi_{i+1}-\Psi_{i}}{\lambda_{i}} = 2\rF\Theta_{i} ,\quad
	\frac{\Theta_{i+1}-\Theta_{i}}{\lambda_{i+1}}=  \Theta_{i+1},
\end{equation}
with the restriction 
\begin{equation}\label{eq:cond-lambdai}
	\lambda_{i}=\alpha  \sqrt{ \Phi_i\Psi_{i}}/ \Theta_{i}\quad \text{for}~i\in\mathbb N,
\end{equation}
where $\alpha\in(0,1/\nm{K}]$.
For later use, we give a useful lemma which provides some technical estimates on the sequence $\{\lambda_i:\,i\in\mathbb N\}$.

\begin{lem}\label{lem:lambdai-icpdps}
	Given $\Theta_0,\Phi_0$ and $\Psi_0>0$
	such that $\Theta_0^2\nm{K}^2=\Phi_0\Psi_0$, 
	the positive sequence $\{\lambda_i:\,i\in\mathbb N\}$ defined by  \cref{eq:para-icpdps,eq:cond-lambdai} satisfies $0<\lambda_{i+1}\leq \lambda_i\leq 1$ and
	\[
	\snm{\lambda_{i+1}-\lambda_i}\leq \left(1+2C_0\right)\lambda_i^2,
	\]
	where
	\begin{equation}\label{eq:c0}
		C_0 := 10+\frac{24\gG\rF}{\nm{K}^2}+\frac{8\gG \Psi_0+8\rF\Phi_0}{\nm{K}^2\Theta_0}.
	\end{equation}
	In particular, we have $  \snm{\lambda_i/\lambda_{i+1}-1}\leq (1+C_0)\lambda_i$.
\end{lem}
\begin{proof} 
	Since $\nm{K}^2\Theta_0^2=\Phi_0\Psi_0$, we have $\lambda_0=1$.
	By \cref{eq:para-icpdps,eq:cond-lambdai}, it follows that
	\begin{align}
		\lambda_{i+1}
		={}&\frac{\sqrt{\Phi_{i+1}\Psi_{i+1}}}{\nm{K}\Theta_{i+1}}= 
		\frac{\sqrt{\Phi_{i+1}\Psi_{i+1}}}{\nm{K}\Theta_{i}}(1-\lambda_{i+1})
		\label{eq:mid}
		\\
		= {}&\frac{1-\lambda_{i+1}}{\nm{K}\Theta_i}\sqrt{(\Phi_i+2\gG \Theta_i\lambda_i)(\Psi_i+ 2\rF\Theta_{i}\lambda_i)}
		={}\sqrt{A_i}\lambda_i(1-\lambda_{i+1}),\notag
	\end{align}
	where 
	\begin{equation}\label{eq:Ai}
		A_i=  1 +\frac{4\gG \rF}{\nm{K}^2}+  \frac{2\gG \Psi_i+2\rF\Phi_i}{\nm{K}\sqrt{\Phi_i\Psi_i}}.
	\end{equation}
	Therefore, given $\Phi_i,\Psi_i,\Theta_i,\lambda_i>0$, $	\lambda_{i+1}	={}  \lambda_i\sqrt{A_i}/(1+\lambda_i\sqrt{A_i})\in(0,1]$ is well determined. 	If we can prove 
	\begin{equation}\label{eq:pf-key-est}
		0\leq  	A_i- 1\leq C_0\lambda_i,
	\end{equation}
	where $C_0>0$ is defined by \cref{eq:c0}, then it follows that 
	\[
	\snm{\lambda_{i+1}-\lambda_i }= 
	\lambda_i\snm{	\frac{1+\sqrt{A_i}(\lambda_i-1)}{1+\lambda_i\sqrt{A_i}}} \leq  \lambda_i\snm{\sqrt{A_i}-1}+\sqrt{A_i}\lambda_i^2\leq \left(C_0+\sqrt{1+C_0}\right)\lambda_{i}^2,
	\]
	and
	\[
	\snm{\lambda_i/\lambda_{i+1}-1} =\snm{\frac{1-\sqrt{A_i}}{\sqrt{A_i}}+\lambda_i}\leq \snm{A_i-1}+\lambda_i\leq (1+C_0)\lambda_i,
	\]
	which implies the desired results immediately. We complete the proof of this lemma and refer to \cref{app:pf-key-est} for a rigorous justification of the key estimate \cref{eq:pf-key-est} and the mononticity property $\lambda_{i+1}\leq \lambda_i$. 
\end{proof}

Now, we define the ansatzes $\phi(t_i) = \Phi_i,\,\psi(t_i) = \Psi_{i},\, \theta(t_{i}) = \Theta_{i}$ with the rescaled time 
\begin{equation}\label{eq:ti}
	t_i = \sum_{j=0}^{i-1}\lambda_j,
\end{equation} 
which means the time step size is $\Delta_i=\lambda_i$. By \cref{eq:para-icpdps,lem:lambdai-icpdps}, applying standard Taylor expansion gives 
\[  	 
\phi'(t_i)={} 2\gG\theta(t_i)+O(\lambda_i),\quad
\psi'(t_i)={} 2\rF\theta(t_i) + O(\lambda_i),\quad
\theta'(t_i)={} \theta(t_i) + O(\lambda_i).  
\]
Letting the time $t_i$ be fixed and taking the limit $\lambda_{i}\to0$, we obtain the continuous parameter equation \cref{eq:ode-para}.
\subsection{The low-resolution ODE}
\label{sec:icpdps-ode}
Introduce the ansatzes 
\[
x(t_i) = x^i,\,y(t_i) = y^{i},\,\zeta(t_i) = \zeta^i,\,\eta(t_i) = \eta^{i}.
\]
Recall the auxiliary sequence $z^i=(\zeta^i,\eta^i)$ defined by \cref{eq:zetai-etai}, which implies that ($\lambda_{i+1}=\mu_{i+1}$)
\begin{equation}\label{eq:diff-x-y}
	\frac{x^{i+1}-x^i}{\lambda_{i}} = {}\zeta^{i+1}-x^i,\quad
	\frac{y^{i+1}-y^i}{\lambda_{i+1}} = {}\eta^{i+1}-y^i. 
\end{equation}
This also refers to lines \ref{algo:zeta} and \ref{algo:eta} in \cref{algo:IC-PDPS}. By \cref{eq:diff-x-y,lem:lambdai-icpdps}, we have
\begin{equation}\label{eq:diff-X-Y}
	x'(t_i) = {}  \zeta(t_{i})- x(t_i) + O(\lambda_i), \quad
	y'(t_{i}) = {} \eta(t_i)-y(t_i)+O(\lambda_i).
\end{equation}

On the other hand, expand the proximal steps in lines \ref{algo:x} and \ref{algo:y} into inclusions gives 
\begin{equation}\label{eq:diff-zeta-eta}
	\left\{
	\begin{aligned}	
		{}& 		\Phi_i   \frac{\zeta^{i+1}-\zeta^i}{\lambda_i} =\Theta_{i}\left[\gG ( x^{i+1}-\zeta^{i+1})-\nabla G(x^{i+1})-K^*\eta^i\right],\\
		{}&	\Psi_{i+1} \frac{\eta^{i+1}-\eta^i}{\lambda_{i+1}} = \Theta_{i+1}\left[\rF( y^{i+1}-\eta^{i+1})-\nabla F^*(y^{i+1})+K\bar \zeta^{i+1}\right],
	\end{aligned}
	\right.
\end{equation}
where $\bar \zeta^{i+1}=\zeta^{i+1}+\omega_i(\zeta^{i+1}-\zeta^i)$ and $\omega_i = \lambda_i/\lambda_{i+1}-\lambda_i$.
For a detailed proof of \cref{eq:diff-zeta-eta}, we refer to \cref{app:pf-diff-icpdps}. Thanks to \cref{lem:lambdai-icpdps}, it holds that
\[
\snm{\omega_i-1} = \snm{\frac{\lambda_i}{\lambda_{i+1}}-1-\lambda_i}\leq (2+C_0)\lambda_i,
\]
where $C_0>0$ is defined by \cref{eq:c0} and is independent on $\lambda_i$. This means $\bar \zeta^{i+1}=\zeta^{i+1}+\omega_i(\zeta^{i+1}-\zeta^i)=\zeta(t_i)+O(\lambda_i)$.
Hence, by introducing the ansatzes $\zeta(t_i) = \zeta^i$ and $\eta(t_i) = \eta^{i}$, we obtain from \cref{eq:diff-zeta-eta} and standard Taylor expansion that 
\begin{equation}\label{eq:diff-Zeta-Eta}
	\left\{
	\begin{aligned} 
		{}&\frac{\phi(t_i)}{ \theta(t_{i})}\cdot  \zeta'(t_i)= \gG ( x(t_{i})-\zeta(t_{i}))-\nabla G(x(t_{i}))-K^*\eta(t_{i})	  +O(\lambda_i),\\		
		{}&\frac{\psi(t_i)}{\theta(t_{i})}\cdot \eta'(t_i) = \rF( y(t_{i})-\eta(t_{i}))-\nabla F^*(y(t_{i}))+K\zeta(t_{i}) 	  +O(\lambda_i).
	\end{aligned}
	\right.
\end{equation}
In view of \cref{eq:diff-X-Y,eq:diff-Zeta-Eta}, fixing the time $t_i$ and taking the limit $\lambda_{i}\to0$, we get  the desired low-resolution ODE \cref{eq:ode-ic-pdps}.
\begin{rem}
	Note that \cref{eq:diff-x-y,eq:diff-zeta-eta,eq:para-icpdps} gives a closed iteration scheme with respect to the sequences $\{x^i,y^i,\zeta^i,\eta^i:\,i\in\mathbb N\}$ and the parameters $\{\Phi_i,\Psi_i,\Theta_{i}:\,i\in\mathbb N\}$.
	As a byproduct, this claims that IC-PDPS actually corresponds to a semi-implicit Euler discretization of the continuous model \cref{eq:ode-ic-pdps} and the parameter equation \cref{eq:ode-para}.
\end{rem}

		\subsection{Lyapunov analysis}
		\label{eq:lyap-res-icpdps-ode}
		In this part, we present a Lyapunov analysis of the continuous model \cref{eq:ode-ic-pdps}. Recalling the alternate Lagrangian defined by \cref{eq:alter-L}, we introduce the following Lyapunov function 
		\begin{equation}\label{eq:lyap}
			\mathcal E(t): = \theta(t)\left[\widehat{\mathcal{L}}(x(t),\widehat{y})-\widehat{\mathcal{L}}(\widehat{x},y(t))\right] + \frac{1}{2}\nm{z(t)-\widehat{u}}^2_{\Upsilon(t)},
		\end{equation}
		where $z(t)=(\zeta(t),\eta(t))=(x(t)+x'(t),y(t)+y'(t))$ and $	\Upsilon(t) =  \diag{ 
			\phi(t)I, \psi(t)I }$. According to the parameter equation \cref{eq:ode-para}, it is easy to obtain
		\[
		\phi(t)=
		\phi(0)+2\gG\theta(0)(e^t-1),\quad\psi(t)=\psi(0)+2\rF\theta(0)(e^t-1),\quad \theta(t)=\theta(0)e^t.
		\]
		The main result is given below, which together with the above fact implies the convergence rate of the Lyapunov function along the continuous trajectory.
		\begin{thm}\label{thm:conv-lyap}
			Let $(x(t),y(t),\zeta(t),\eta(t))$ be the smooth solution to \cref{eq:ode-ic-pdps}.
			The Lyapunov function \cref{eq:lyap} satisfies
			\begin{equation}\label{eq:diff-lyap}
				\frac{\dd}{\dd t}\mathcal E(t)\leq 0,
			\end{equation}
			which implies that 
			\[
			\theta(t)\left[\widehat{\mathcal{L}}(x(t),\widehat{y})-\widehat{\mathcal{L}}(\widehat{x},y(t))\right] + \frac{1}{2}\nm{z(t)-\widehat{u}}^2_{\Upsilon(t)}\leq \mathcal E(0)\quad\forall\,t\geq 0.
			\]
		\end{thm}
		\begin{proof}
			It is sufficient to establish \cref{eq:diff-lyap}. 
			By the definition \cref{eq:alter-L}, we have
			\[
			\widehat{\mathcal{L}}(x(t),\widehat{y})-\widehat{\mathcal{L}}(\widehat{x},y(t))=	{}\mathcal{L}(x(t),\widehat{y})-\mathcal{L}(\widehat{x},y(t))- \frac{\gG}{2}\nm{x(t)-\widehat{x}}-\frac{\rF}{2}\nm{y(t)-\widehat{y}},
			\]
			which implies  
			\begin{equation}\label{eq:grad-hat-L}
				\nabla\left[\widehat{\mathcal{L}}(x(t),\widehat{y})-\widehat{\mathcal{L}}(\widehat{x},y(t))\right]=  H(u(t))
				-\Gamma(u(t)-\widehat{u}) . 
			\end{equation}
			Therefore, it follows that
			\[
			\begin{aligned}
				\mathcal E'(t) ={}& \theta'(t)\left[\widehat{\mathcal{L}}(x(t),\widehat{y})-\widehat{\mathcal{L}}(\widehat{x},y(t))\right]  + \theta(t)\dual{u'(t),H(u(t))
					-\Gamma(u(t)-\widehat{u})} \\
				{}&\quad + \frac{1}{2}\nm{z(t)-\widehat{u}}^2_{\Upsilon'(t)}
				+\dual{z'(t),z(t)-\widehat{u}}_{\Upsilon(t)}.
			\end{aligned}
			\]
			Rewrite \cref{eq:2ndODE-ICPDPS} equivalently as follows
			\begin{equation}
				\label{eq:ODE-ICPDPS}
				\left\{
				\begin{aligned}
					{}&u'(t)=  z(t)-u(t),\\
					{}&	\Upsilon(t) z'(t) = \theta(t)\left[\Gamma (u(t)-z(t)) - H(u(t))\right],
				\end{aligned}
				\right.
			\end{equation}
			which yields 
			\[
			\begin{aligned}
				{}& \frac{1}{2}\nm{z(t)-\widehat{u}}^2_{\Upsilon'(t)}
				+\dual{z'(t),z(t)-\widehat{u}}_{\Upsilon(t)}\\ 
				={}&\frac{1}{2}\nm{z(t)-\widehat{u}}^2_{\Upsilon'(t)}
				+\theta(t)\dual{		\Gamma(u(t)-z(t)) - H(u(t)),z(t)-\widehat{u}}.
			\end{aligned}
			\]
			Hence, we obtain 
			\[
			\begin{aligned}
				\mathcal E'(t) ={}& \theta'(t)\left[\widehat{\mathcal{L}}(x(t),\widehat{y})-\widehat{\mathcal{L}}(\widehat{x},y(t))\right]  +\theta(t) \dual{z(t)-u(t),H(u(t))
					-\Gamma(u(t)-\widehat{u})} \\
				{}&\quad +\frac{1}{2}\nm{z(t)-\widehat{u}}^2_{\Upsilon'(t)}
				+\theta(t)\dual{		\Gamma(u(t)-z(t)) - H(u(t)),z(t)-\widehat{u}}\\
				={}& \theta'(t)\left[\widehat{\mathcal{L}}(x(t),\widehat{y})-\widehat{\mathcal{L}}(\widehat{x},y(t))\right]  +\theta(t)	\underbrace{\dual{\widehat{u}-u(t),H(u(t))
						-\Gamma(u(t)-\widehat{u})}}_{\mathbb I_1}\\
				{}&\quad+\underbrace{\frac{1}{2}\nm{z(t)-\widehat{u}}^2_{\Upsilon'(t)}-\theta(t) \dual{\Gamma(z(t)-\widehat{u}),z(t)-\widehat{u}}}_{\mathbb I_2}.
			\end{aligned}
			\]
			
			Let us focus on the two cross terms $\mathbb I_1$ and $\mathbb I_2$. Set $D:=\diag{ 
				\gG I, \rF I }$. Then 
			\[
			\Gamma = D  + \begin{bmatrix}
				O&K^*\\-K&O
			\end{bmatrix}.
			\]
			Since the second part is asymmetric, it follows that
			\[
			\mathbb I_2=	\frac{1}{2}\nm{z(t)-\widehat{u}}^2_{\Upsilon'(t)}-\theta(t) \dual{\Gamma(z(t)-\widehat{u}),z(t)-\widehat{u}} = 	\frac{1}{2}\nm{z(t)-\widehat{u}}^2_{\Upsilon'(t)-2D\theta(t)}.
			\]
			Besides, in view of \cref{eq:grad-hat-L}, we get 
			\[
			\begin{aligned}
				\mathbb I_1
				=	{}&	 \dual{\widehat{u}-u(t),\nabla \left[\widehat{\mathcal{L}}(x(t),\widehat{y})-\widehat{\mathcal{L}}(\widehat{x},y(t))\right]}\\ 
				={}&\dual{\widehat{x}-x(t),\nabla_x \widehat{\mathcal{L}}(x(t),\widehat{y})  }
				-\dual{\widehat{y}-y(t),\nabla_y \widehat{\mathcal{L}}(\widehat{x},y(t))  }\\
				\leq {}&\widehat{\mathcal{L}}(\widehat{x},\widehat{y})-\widehat{\mathcal{L}}(x(t),\widehat{y})
				+\widehat{\mathcal{L}}(\widehat{x},y(t))-	\widehat{\mathcal{L}}(\widehat{x},\widehat{y}) \\
				={}& \widehat{\mathcal{L}}(\widehat{x},y(t)) -\widehat{\mathcal{L}}(x(t),\widehat{y}).
			\end{aligned}
			\]
			Combining this with the last identity for $\mathbb I_2$ gives 
			\begin{equation}\label{eq:key-conv-lyap}
				\mathcal E'(t) 
				\leq  \left(\theta'(t)-\theta(t)\right)\left[\widehat{\mathcal{L}}(x(t),\widehat{y})-\widehat{\mathcal{L}}(\widehat{x},y(t)) \right]+	\frac{1}{2}\nm{z(t)-\widehat{u}}^2_{\Upsilon'(t)-2D\theta(t)}.
			\end{equation}
			Thanks to \cref{eq:ode-para}, we have $\Upsilon'(t)=2D\theta(t)$ and the right hand side vanishes. Consequently, we obtain \cref{eq:diff-lyap} and complete the proof of this theorem.
		\end{proof}
		\begin{rem}
			According to the key estimate \cref{eq:key-conv-lyap}, we claim that the parameter equation \cref{eq:ode-para} can be relaxed as inequalities
			\begin{equation}\label{eq:ode-para-ineq}	    
				\phi'(t)\leq 2\gG\theta(t),\quad
				\psi'(t)\leq2\rF\theta(t),\quad
				\theta'(t)\leq \theta(t), 
			\end{equation}
			which preserves the nonincreasing property \cref{eq:diff-lyap}.
		\end{rem}

\section{The Intrinsic IC-PDPS ODE}
\label{sec:int-ode-icpdps}
In the last section, we derived the continuous model of \cref{algo:IC-PDPS} with respect to the rescaled step size $\lambda_i$. In this part, we shall present another model by using the intrinsic step size $\alpha\in(0,1/\nm{K}]$. The resulted ODE \cref{eq:int-ode-ic-pdps} is equivalent to the previous one \cref{eq:ode-ic-pdps} under proper time transformation.
\subsection{Deriving the continuous model}
Motivated by \cref{eq:cond-lambdai}, we consider the intrinsic step size $\Delta_i=\alpha\in(0,1/\nm{K}]$.
Then \cref{eq:para-icpdps} becomes 
\begin{equation}\label{eq:int-para}
	\small
	\frac{\Phi_{i+1}-\Phi_i}{\alpha } = 2\gG \sqrt{\Phi_i\Psi_i},\quad
	\frac{\Psi_{i+1}-\Psi_{i}}{\alpha } = 2\rF\sqrt{\Phi_i\Psi_i} ,\quad
	\frac{\Theta_{i+1}-\Theta_{i}}{\alpha }= \sqrt{\Phi_{i+1}\Psi_{i+1}}.
\end{equation}
The iterative sequences \cref{eq:diff-x-y} reads equivalently as 
\begin{equation}\label{eq:diff-x-y-int}
	\frac{x^{i+1}-x^i}{\alpha } = {}\frac{\sqrt{\Phi_i\Psi_i}}{\Theta_i}\left(\zeta^{i+1}-x^i\right),\quad
	\frac{y^{i+1}-y^i}{\alpha } = {}\frac{\sqrt{\Phi_{i+1}\Psi_{i+1}}}{\Theta_{i+1}}\left(\eta^{i+1}-y^i\right),
\end{equation}
and \cref{eq:diff-zeta-eta} becomes 
\begin{equation}\label{eq:diff-zeta-eta-int}
	\left\{
	\begin{aligned}	
		{}& 	\sqrt{	 \Phi_i   } \frac{\zeta^{i+1}-\zeta^i}{\alpha} = \sqrt{\Psi_i}\left[\gG ( x^{i+1}-\zeta^{i+1})-\nabla G(x^{i+1})-K^*\eta^i\right] ,\\
		{}&	\sqrt{ \Psi_{i+1}} \frac{\eta^{i+1}-\eta^i}{\alpha } = \sqrt{\Phi_{i+1} } \left[ \rF( y^{i+1}-\eta^{i+1})-\nabla F^*(y^{i+1})+K\bar \zeta^{i+1}\right].
	\end{aligned}
	\right.
\end{equation}

Define the intrinsic time $s_i = \alpha  i$ and the ansatzs
\[
\widetilde{\phi}(s_i) = \Phi_i,\,\widetilde{\psi}(s_i) = \Psi_i,\,\widetilde{\theta}(s_i) = \Theta_i.
\]
Taking the limit $\alpha \to 0$ in \cref{eq:int-para} leads to
\begin{equation}\label{eq:res-ode-para} 
	{}\widetilde{\phi}'(s)= 2\gG\sqrt{\widetilde{\phi}(s)\widetilde{\psi}(s)},\quad
	{} \widetilde{\psi}'(s)=2\rF\sqrt{\widetilde{\phi}(s)\widetilde{\psi}(s)},\quad
	{} \widetilde{\theta}'(s)= \sqrt{\widetilde{\phi}(s)\widetilde{\psi}(s)}.
\end{equation}
Similarly, by introducing the ansatzs
\[
\widetilde{x}(s_i) = x^i,\quad \widetilde{y}(s_i) = y^i,
\quad\widetilde{\zeta}(s_i) = \zeta^i,\quad \widetilde{\eta}(s_i) = \eta^i,
\]
we obtain from \cref{eq:diff-x-y-int,eq:diff-zeta-eta-int} that (letting $\alpha\to 0$)
\begin{equation}\label{eq:int-ode-ic-pdps}
	\small
	\left\{
	\begin{aligned}
		{}&\widetilde{\theta}(s)\widetilde{x}'(s)= \widetilde{\theta}'(s)\left[ \widetilde{\zeta}(s)-\widetilde{x}(s)\right],\\
		{}& \widetilde{\theta}(s) \widetilde{y}'(s)=  \widetilde{\theta}'(s)\left[\widetilde{\eta}(s)-\widetilde{y}(s)\right],\\ 
		{}&	\sqrt{ \widetilde{\phi}(s)}  \widetilde{\zeta}'(s) =  \sqrt{\widetilde{\psi}(s)}\left[\gG(\widetilde{x}(s)-\widetilde{\zeta}(s))-\nabla G(\widetilde{x}(s))-K^*\widetilde{\eta}(s) \right],\\
		{}&	\sqrt{ \widetilde{\psi}(s)}  \widetilde{\eta}'(s) =  \sqrt{\widetilde{\phi}(s)}\left[\rF(\widetilde{y}(s)-\widetilde{\eta}(s))-\nabla F^*(\widetilde{y}(s))+K\widetilde{\zeta}(s) \right].
	\end{aligned}
	\right.
\end{equation}
We call this the {\it intrinsic IC-PDPS ODE} since it is derived from the intrinsic step size $\Delta_i=\alpha$.

Note that \cref{eq:int-ode-ic-pdps,eq:res-ode-para} are equivalent to \cref{eq:ode-ic-pdps,eq:ode-para} respectively, under the rescaling transformation
\begin{equation}\label{eq:time-t-ic-pdps}
	t(s) = \ln \widetilde{\theta}(s) - \ln \widetilde{\theta}(0),
\end{equation}
which means
\begin{equation}\label{eq:int-res-para}
	{}\widetilde{\phi} (s)= \phi(t(s)) ,\quad
	{} \widetilde{\psi}(s)=  \psi (t(s)),\quad
	{} \widetilde{\theta}(s)=   \theta (t(s)),
\end{equation}
and 
\begin{equation}\label{eq:int-res-X-Y}
	{}\widetilde{x} (s)= x(t(s)) ,\quad
	{} \widetilde{y}(s)=  y (t(s)),\quad
	{} \widetilde{\zeta}(s)=   \zeta (t(s)),\quad 
	{} \widetilde{\eta}(s)=   \eta (t(s)).
\end{equation}
For example, let us verify $ \widetilde{\theta}(s)=\theta(t(s)) $ for all $s\geq 0$. By definition, it is clear that $t(0)= 0$ and $ \widetilde{\theta}(0)=\theta_0=\theta(0)=\theta(t(0)) $. Moreover, we have
\[
\frac{\dd}{\dd s} \theta(t(s)) =   \theta'(t(s))\frac{\dd}{\dd s}  t(s)\overset{\text{by \cref{eq:time-t-ic-pdps}}}{=}\frac{\widetilde{\theta}'(s)}{\widetilde{\theta}(s)}\theta'(t(s))
\overset{\text{by \cref{eq:ode-para}}}{=} \theta(t(s))\frac{\widetilde{\theta}'(s)}{\widetilde{\theta}(s)}.
\]
Multiplying both sides by $1/\theta(t(s))$ and integrating from 0 to $s$, we get
\[
\ln \theta(t(s))-\ln \theta(t(0)) = \ln \widetilde{\theta}(s)-\ln \widetilde{\theta}(0)\quad\Longrightarrow  \quad \widetilde{\theta}(s)=\theta(t(s)) .
\]
The left relations in \cref{eq:int-res-para,eq:int-res-X-Y} can be proved similarly.
\begin{figure}[H]
	\centering
	\subfigure{
		\includegraphics[width=5cm]{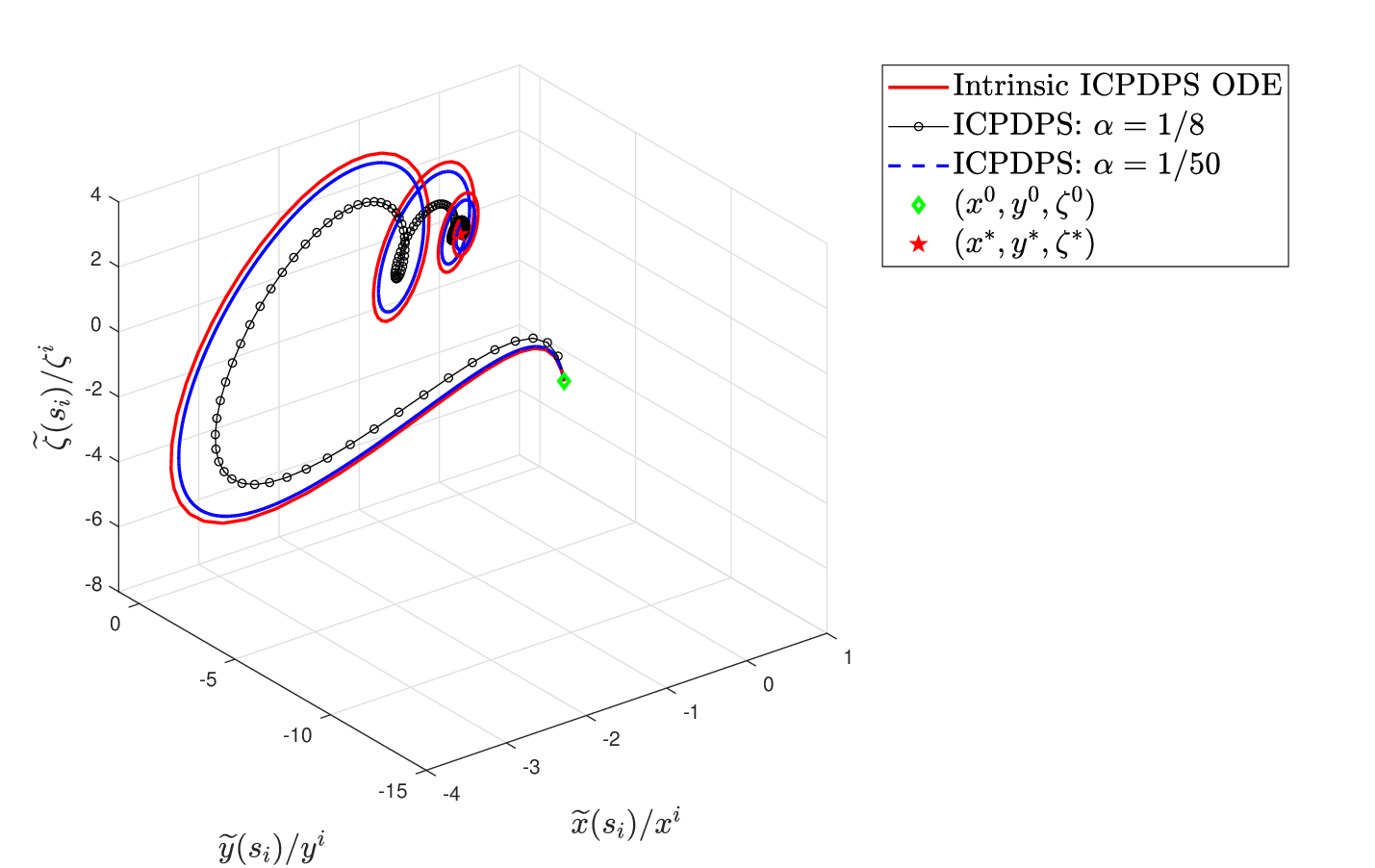}
		\label{fig:icpdps-intode-xyzeta}
	}
	\subfigure{
		\includegraphics[width=5cm]{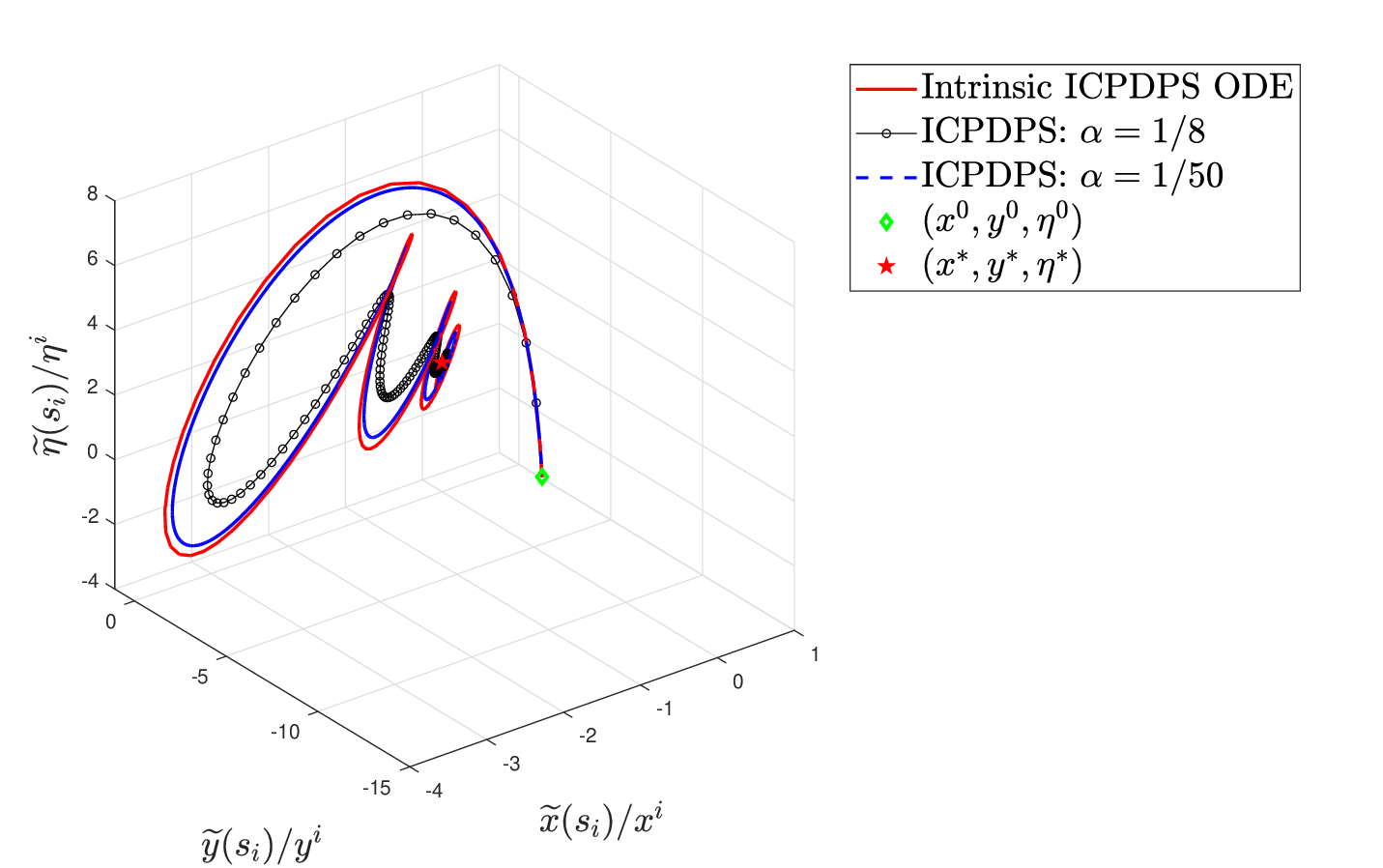}
		\label{fig:icpdps-intode-xyeta}
	}
	\quad    
	\subfigure{
		\includegraphics[width=4cm]{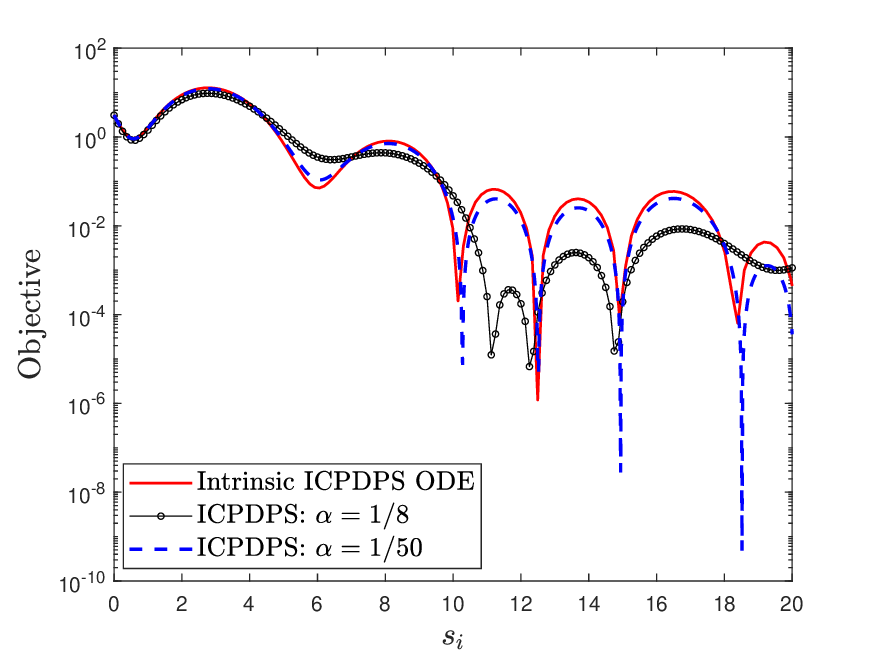}
		\label{fig:icpdps-intode-rate}
	}
	\caption{Numerical illustration of the trajectories $X^i=(x^i,y^i,\zeta^i,\eta^i)$ for  \cref{algo:IC-PDPS} and  $\widetilde{X}(s)=(\widetilde{x}(s),\widetilde{y}(s),\widetilde{\zeta}(s),\widetilde{\eta}(s))$ for the intrinsic IC-PDPS ODE \cref{eq:int-ode-ic-pdps}. We simply take $K=1$ and quadratic objectives: $G(x)=x^2/2,\,F^*(y) = y^2/2$ in one dimension.  From the left and middle, we see smaller $\alpha$ leads to better approximations of the continuous trajectories. In the right, we also report the convergence behavior of the objective in terms of the time $s_i=\alpha i\in[0,20]$.}
	\label{fig:icpdps-intode}
\end{figure}
\begin{figure}[H]
	\centering
	\subfigure{
		\includegraphics[width=5cm]{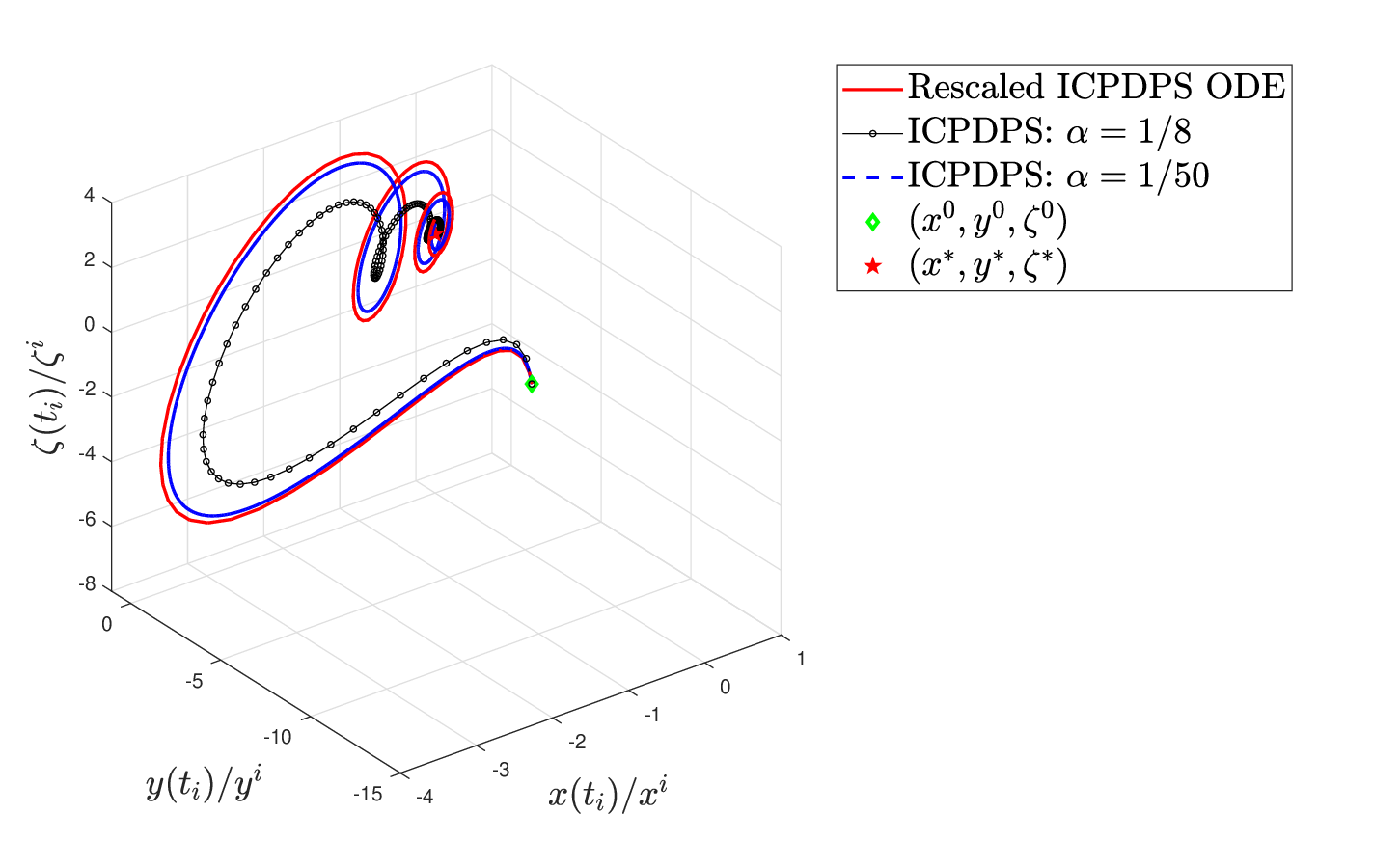}
		\label{fig:icpdps-resode-xyzeta}
	}
	\subfigure{
		\includegraphics[width=5cm]{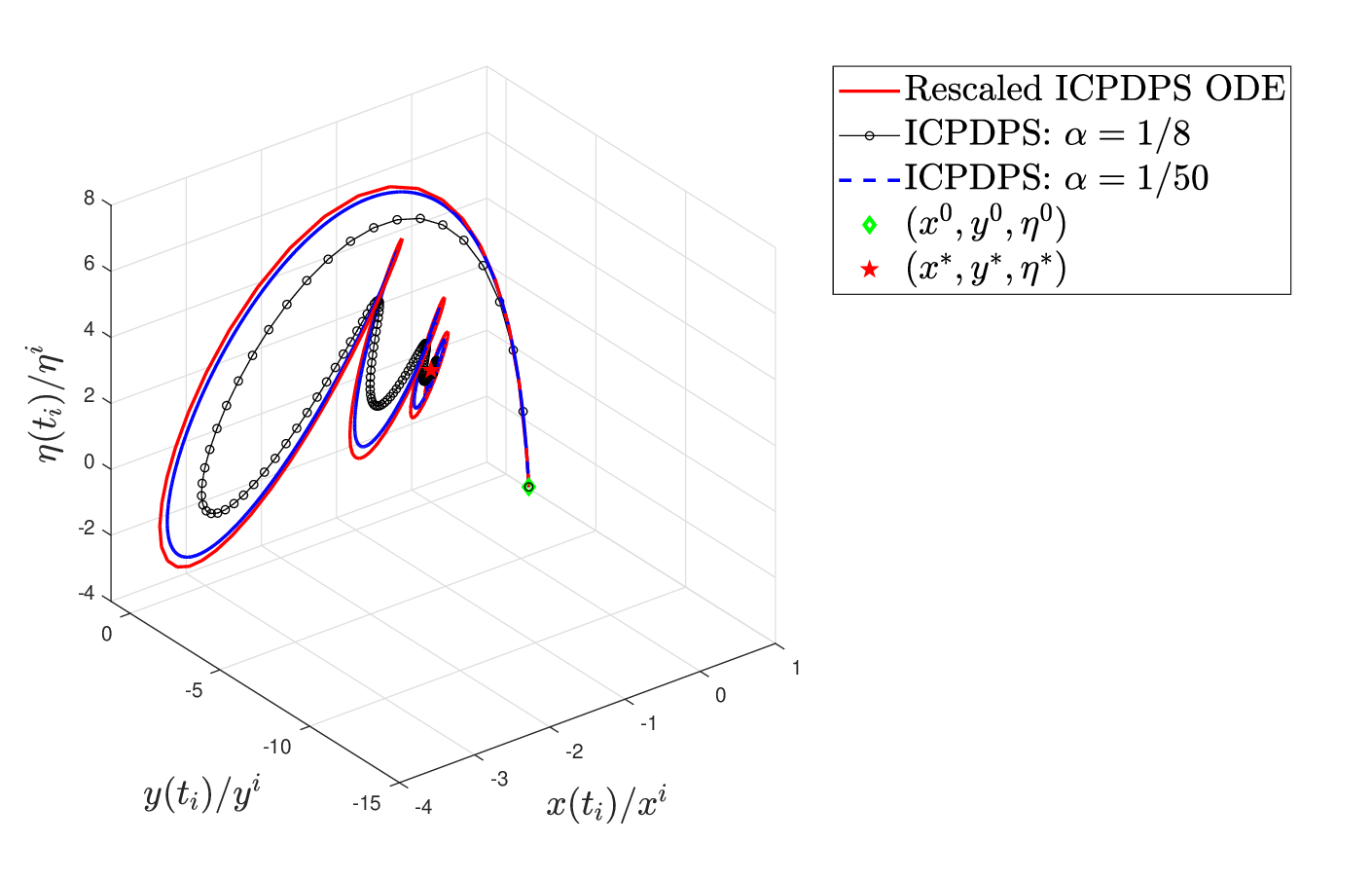}
		\label{fig:icpdps-resode-xyeta}
	}
	\quad    
	\subfigure{
		\includegraphics[width=4cm]{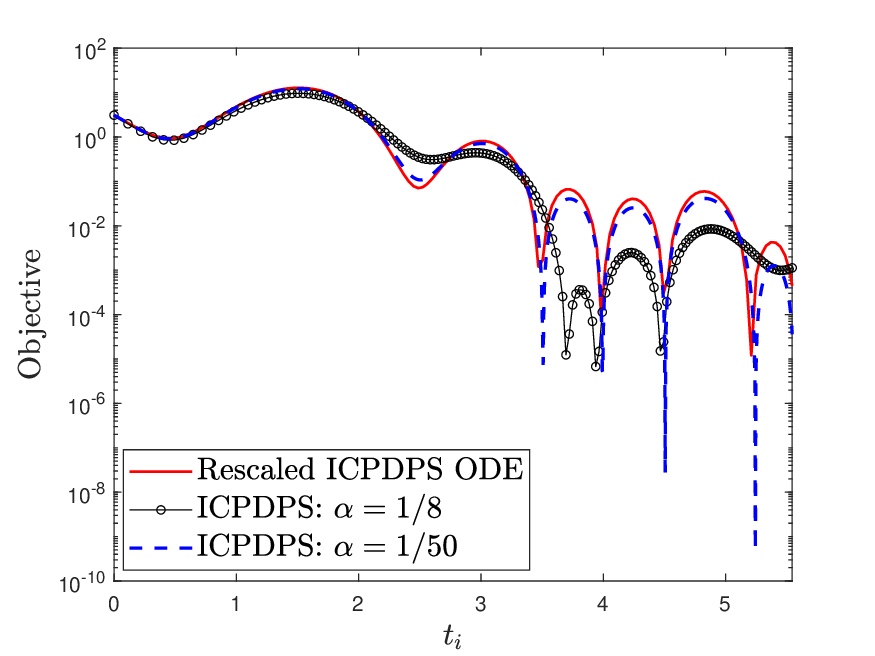}
		\label{fig:icpdps-resode-rate}
	}
	\caption{Numerical illustration of the trajectories $X^i=(x^i,y^i,\zeta^i,\eta^i)$ for \cref{algo:IC-PDPS} and  $X(t)=(x(t),y(t),\zeta(t),\eta(t))$ for the rescaled IC-PDPS ODE \cref{eq:ode-ic-pdps}. The objectives are similar with that in \cref{fig:icpdps-intode}. From the left and middle, we observe that $X(t)$ and $\widetilde{X}(s)$ are identical in the phase space. In the right, we show the convergence behavior of the objective with respect to the rescaled time $t_i$, which is related to the intrinsic time $s_i$ by the transformation \cref{eq:time-t-ic-pdps}.}
	\label{fig:icpdps-resode}
\end{figure} 
\subsection{Lyapunov analysis}
In this part, we present a Lyapunov analysis of the intrinsic IC-PDPS ODE \cref{eq:int-ode-ic-pdps}. Introduce the following Lyapunov function 
\begin{equation}\label{eq:lyap-int}
	\widetilde{	\mathcal E}(s): = \widetilde{\theta}(s)\left[\widehat{\mathcal{L}}(\widetilde{x}(s),\widehat{y})-\widehat{\mathcal{L}}(\widehat{x},\widetilde{y}(s))\right] + \frac{1}{2}\nm{\widetilde{z}(s)-\widehat{u}}^2_{\widetilde{\Upsilon}(s)},
\end{equation}
where $\widetilde{z}(s)=(\widetilde{\zeta}(s),\widetilde{\eta}(s))$ and $	\widetilde{\Upsilon}(s) =  {\rm diag}(
\widetilde{\phi}(s)I, \widetilde{\psi}(s)I )$. The main result is given below.
\begin{thm}\label{thm:conv-lyap-int}
	Let $(\widetilde{x}(s),\widetilde{y}(s),\widetilde{\zeta}(s),\widetilde{\eta}(s))$ be the smooth solution to \cref{eq:int-ode-ic-pdps}.
	The Lyapunov function \cref{eq:lyap-int} satisfies
	\begin{equation}\label{eq:conv-lyap-int}
		\widetilde{\theta}(s)\left[\widehat{\mathcal{L}}(\widetilde{x}(s),\widehat{y})-\widehat{\mathcal{L}}(\widehat{x},\widetilde{y}(s))\right] + \frac{1}{2}\nm{\widetilde{z}(s)-\widehat{u}}^2_{\widetilde{\Upsilon}(s)}\leq \widetilde{	\mathcal E}(0)\quad\forall\,s\geq 0.
	\end{equation}
	Moreover, we have the lower bound estimate
	\begin{equation} 
		\label{eq:pf-lb-Theta-int}
		\widetilde{\theta}(s)\geq  	\sqrt{\frac{\widetilde{\phi}(0)\widetilde{\psi}(0)}{3}}s +\frac{\gG\widetilde{\psi}(0)+\rF\widetilde{\phi}(0)}{6}s^2+\widetilde{\theta}(0)\exp\left(  \sqrt{\frac{\gG\rF}{3 }} 2s \right),
	\end{equation}
	for all $s\geq 0$.
\end{thm}
\begin{proof}
	Observing \cref{eq:int-res-para,eq:int-res-X-Y}, we claim that $
	\widetilde{\mathcal E}(s) =\mathcal E(t(s))$ and it follows
	\[
	\frac{\dd}{\dd s}\widetilde{\mathcal E}(s) =\mathcal E'(t(s))\cdot 	\frac{\dd}{\dd s}t(s) \overset{\text{by \cref{eq:time-t-ic-pdps}}}{=}\mathcal E'(t(s)) \frac{\widetilde{\theta}'(s)}{\widetilde{\theta}(s)}\overset{\text{by \cref{eq:res-ode-para}}}{=}\frac{ \sqrt{\widetilde{\phi}(s)\widetilde{\psi}(s)}}{\widetilde{\theta}(s)}\mathcal E'(t(s))\overset{\text{by \cref{eq:diff-lyap}}}{\leq} 0,
	\]
	which implies \cref{eq:conv-lyap-int}.
	
	It remains to prove \cref{eq:pf-lb-Theta-int}. Notice that by \cref{eq:res-ode-para}, we have
	\[
	\widetilde{\phi}(s) =	\widetilde{\phi}(0)+ 2\gG(\widetilde{\theta}(s)-\widetilde{\theta}(0)),\quad 	\widetilde{\psi}(s) =	\widetilde{\psi}(0)+ 2\rF(\widetilde{\theta}(s)-\widetilde{\theta}(0)).
	\]
	For ease of notation, let $R(s) = \widetilde{\theta}(s)-\widetilde{\theta}(0)$. Then it holds that 
	\[
	\begin{aligned}
		R'(s)=	{}&	\widetilde{\theta}'(s) = \sqrt{	\widetilde{\phi}(s)  	\widetilde{\psi}(s)}=\sqrt{\widetilde{\phi}(0)+ 2\gG R(s)}\cdot \sqrt{\widetilde{\psi}(0)+ 2\rF R(s)}\\
		={}&\sqrt{\widetilde{\phi}(0)\widetilde{\psi}(0)+2(\gG\widetilde{\psi}(0)+\rF\widetilde{\phi}(0))R(s)+4\gG\rF R^2(s) }\\
		\geq {}&\frac{1}{\sqrt{3}}\left(\sqrt{\widetilde{\phi}(0)\widetilde{\psi}(0)}+\sqrt{2\gG\widetilde{\psi}(0)+2\rF\widetilde{\phi}(0)}\sqrt{R(s)}+2\sqrt{\gG\rF} R(s) \right).
	\end{aligned}
	\]
	Define $Q(s):=		Q_1(s)+		Q_2(s)+Q_3(s)$ with
	\[
	\begin{aligned}
		Q_1(s)
		={}&\sqrt{\frac{\widetilde{\phi}(0)\widetilde{\psi}(0)}{3}}s,\quad Q_2(s)=\frac{\gG\widetilde{\psi}(0)+\rF\widetilde{\phi}(0)}{6}s^2,\\
		Q_3(s)	={}&\widetilde{\theta}(0)\exp\left(  \sqrt{\frac{\gG\rF}{3 }} 2s \right)-\widetilde{\theta}(0).
	\end{aligned}
	\]		
	Observe that $R(0)=Q(0)=0$ and 
	\[
	\begin{aligned}
		Q'(s) ={}& Q'_1(s)  + Q'_2(s)+Q_3'(s) \\ 
		={}&\frac{1}{\sqrt{3}}\left(\sqrt{\widetilde{\phi}(0)\widetilde{\psi}(0)}+\sqrt{2\gG\widetilde{\psi}(0)+2\rF\widetilde{\phi}(0)}\sqrt{Q_2(s)}+2\sqrt{\gG\rF} Q_3(s) \right)\\
		\leq {}&\frac{1}{\sqrt{3}}\left(\sqrt{\widetilde{\phi}(0)\widetilde{\psi}(0)}+\sqrt{2\gG\widetilde{\psi}(0)+2\rF\widetilde{\phi}(0)}\sqrt{Q(s)}+2\sqrt{\gG\rF} Q(s) \right).
	\end{aligned}
	\]
	Therefore, invoking the comparison principle \cite{McNabb1986}, we obtain $R(s)\geq Q(s)$, which implies \cref{eq:pf-lb-Theta-int} and  completes the proof of this theorem.
\end{proof}
\section{Concluding Remarks}
\label{sec:conclu}
In this work, we give a continuous perspective on the inertial corrected primal-dual proximal splitting (IC-PDPS). To handle the complicated combination of iterative sequences and parameters, we introduce the idea of rescaled step size and reformulate IC-PDPS as a collection of first-order finite difference templates. The continuous-time limit admits the same second-order in time feature as many accelerated ODEs but involves additional symmetric velocity correction. We also propose tailored Lyapunov functions and establish the convergence rate of the Lagrangian gap.

This alternative point of view of IC-PDPS also suggests some interesting future topics. Let us discuss more as follows.

\vskip0.1cm\noindent{\bf Velocity correction.} 
Recall the rescaled IC-PDPS ODE in component wise (cf.\cref{eq:ic-pdps-2nd-ode})
\[
\left\{
\begin{aligned}
	{}&	\frac{ 	\phi(t)}{\theta(t)} x''(t)+\left[\gG+\frac{ 	\phi(t)}{\theta(t)} \right]x'(t)  + \nabla G(x(t))+K^*(y(t)+y'(t))= 0,\\
	{}&	\frac{ 	\psi(t)}{\theta(t)} y''(t)+\left[\rF+\frac{ 	\psi(t)}{\theta(t)} \right]y'(t)  + \nabla F^*(y(t))-K(x(t)+x'(t))= 0.	
\end{aligned}
\right.
\]
When $K$ vanishes, the primal and dual variables are decoupled and the continuous model is very close to the NAG flow in \cite{Luo2022d}. However, the existence of the velocity correction for general case $K\neq O$ looks very subtle. Following the analysis for continuous-time Chambolle-Pock \cite{luo_primal-dual_2022}, some further efforts (such as spectral analysis) shall be paid to understand this.

\vskip0.1cm\noindent{\bf Non-Euclidean setting.} 
Note that both IC-PDPS and its continuous models derived in this paper consider only the standard inner product (Euclidean setting). It deserves proper extensions under general metrics such as Bregman distances and preconditioners.

\vskip0.1cm\noindent{\bf High-resolution model.}
According to the high-resolution analysis \cite{shi_understanding_2021}, Nesterov acceleration involves hidden second-order information as gradient correction which makes it more stable than the Heavy-ball method. It is interesting to see whether this exists in the primal-dual setting. This also motivates us to deriving the high-resolution IC-PDPS ODE with Hessian-driven damping.

\vskip0.1cm\noindent{\bf Long time error estimate.} The current work provides only the Lyapunov analysis of the continuous trajectory $X(t)$. This actually implies the convergence behavior for the long time case $t\to\infty$. Taking the time discretization presentations \cref{eq:diff-x-y-int,eq:diff-zeta-eta-int} into account, one can consider the numerical analysis aspect and establish the error estimate with fixed time $T>0$:
\begin{equation}\label{eq:error}
	\max_{1\leq i\leq N} \nm{X(t_i)-X^i}\leq C(T)\alpha,\quad N = T/\alpha,
\end{equation}
where $C(T)>0$ depends on $T$. This has at least been verified numerically by our simple examples in \cref{fig:icpdps-intode,fig:icpdps-resode}. Since we are care more about the limiting case $T\to\infty$, a stronger version of \cref{eq:error} is important. This leads to the {\it long time error estimate conjecture}
\[
\nm{X(t_i)-X^i}\leq C\alpha,\quad \forall\,i\in\mathbb N,
\]
where $C>0$ is independent of $i$ and $\alpha$. We left these interesting topics as our future works.

\begin{appendices}
	\section{Missing Proofs in \cref{sec:NAG}}
	\subsection{Sharp estimate of the parameter $\lambda_i$ in NAG}
	\label{app:pf-ub-lb-lambdai}
	\begin{lem}\label{lem:ub-lb-lambdai}
		Given $\lambda_0>0$, the positive sequence $\{\lambda_i:\,i\in\mathbb N\}$ defined in \cref{eq:NAG} satisfies $0\leq \lambda_{i}-\lambda_{i+1}\leq \lambda_i^2$ and 
		\begin{equation}\label{eq:ub-lb-lambdai}
			\frac{2}{i+2/\lambda_0+C(\lambda_0,i)}\leq 	\lambda_i\leq \frac{2}{i+2/\lambda_0} \quad\forall\,i\in\mathbb N,
		\end{equation}
		where $C(\lambda_0,i):=\lambda_0/4+1/2\ln(1+\lambda_0i/2)$. 
	\end{lem}
	\begin{proof}
		Since $\lambda_{i+1}^{-1}  =1/2+   \sqrt{\lambda_{i}^{-2}+1/4}$, it follows that 
		\[
		\begin{aligned}
			\lambda_{i+1}-\lambda_i ={}&  \frac{2\lambda_i-\lambda_i^2-\lambda_i\sqrt{4+\lambda_i^2}}{\lambda_i+\sqrt{4+\lambda_i^2}}=-\frac{2\lambda_i^2}{2+\lambda_i+\sqrt{4+\lambda_i^2}},
		\end{aligned}
		\]
		which yields $0\leq \lambda_{i}-\lambda_{i+1}\leq\lambda_i^2$. Besides, we have
		\[
		\lambda_{i+1}^{-1}-\lambda_i^{-1} = \frac{\lambda_i+\sqrt{4+\lambda_i^2}}{2+\lambda_i+\sqrt{4+\lambda_i^2}} = \frac{1}{2} + \frac{\lambda_i}{4+2\sqrt{4+\lambda_i^2}}.
		\]
		For any $N\geq 1$, taking summation over $0\leq i\leq N-1$ gives
		\[
		\lambda_N^{-1}-\lambda_0^{-1} = \frac{N}{2}+\sum_{i=0}^{N-1}\frac{\lambda_i}{4+2\sqrt{4+\lambda_i^2}}.
		\]
		On the one hand, this implies 
		\[
		\lambda_N^{-1}-\lambda_0^{-1} \geq \frac{N}{2} \quad\Longrightarrow\quad \lambda_N\leq \frac{2}{N+2/\lambda_0}.
		\]
		In addition, we see that
		\[
		\lambda_N^{-1}-\lambda_0^{-1} \leq  \frac{N}{2}+\sum_{i=0}^{N-1}\frac{\lambda_i}{8}\leq \frac{N}{2}+\frac1{4}\sum_{i=0}^{N-1}\frac{1}{i+2/\lambda_0}
		\leq \frac{N}{2}+\frac{1}{4} \left(\lambda_0/2+\ln(1+\lambda_0N/2)\right).
		\]
		Hence, the lower bound follows directly
		\[
		\lambda_N\geq 		\frac{2}{N+2/\lambda_0+C(\lambda_0,N)},
		\]
		where $C(\lambda_0,N)=\lambda_0/4+1/2\ln(1+\lambda_0N/2)$. 
		This finishes the proof of \cref{lem:ub-lb-lambdai}.
	\end{proof}
	
	\subsection{Derivation of \cref{eq:NAG-int-ode}}
	\label{app:NAG-int-ode}
	Recall the intrinsic time $
	s_i = i\sqrt{\tau}$ and the ansatz 
	\[
	\widetilde{X}(s_i) = (\widetilde{x}(s_i ),\, \widetilde{{z}}(s_i ),\, \widetilde{\theta}(s_i )) = (x^i,z^i,\sqrt{\tau}/\lambda_i).
	\]
	Clearly, one has
	\[
	\frac{\widetilde{X}(s_{i+1})-\widetilde{X}(s_i)}{\sqrt{\tau}} = \widetilde{X}'(s_i) + O(\sqrt{\tau}).
	\]
	By \cref{eq:barxi} we have 
	\[
	\begin{aligned}
		\bar x^i={}&x^{i+1}+\lambda_i(z^i-z^{i+1})=\widetilde{x}(s_{i+1})+\sqrt{\tau}(\widetilde{z}(s_i)-\widetilde{z}(s_{i+1}))/\widetilde{\theta}(s_i)
		={}\widetilde{x}(s_i)+O(\sqrt{\tau}),
	\end{aligned}
	\]
	which implies $		\nabla G(\bar x^i) 
	={}\nabla G(\widetilde{x}(s_i) ) +O(\sqrt{\tau})$. 
	Plugging this into \cref{eq:intrinsic-NAG-ode} gives 
	\[
	\left\{
	\begin{aligned}
		{}& \widetilde{z}'(s_i)=- \widetilde{\theta}(s_i ) \nabla G(\widetilde{x}(s_i) ) +O(\sqrt{\tau}) , \\
		{}&\widetilde{\theta}(s_i )\widetilde{x}'(s_i)=\widetilde{z}(s_i)-\widetilde{x}(s_i) +O(\sqrt{\tau}),\\
		{}&2\widetilde{\theta}(s_i )\widetilde{\theta}'(s_i )=\widetilde{\theta}(s_i )+O(\sqrt{\tau}).
	\end{aligned}
	\right.
	\]
	Consequently, letting $\sqrt{\tau}\to0$ implies the continuous model \cref{eq:NAG-int-ode}.
	\subsection{Derivation of \cref{eq:NAG-res-ode}}
	\label{app:NAG-res-ode}
	Let $i\in\mathbb N$ and $t_i=\sum_{j=0}^{i-1}\lambda_j$ be fixed. Using Taylor expansion yields 
	\[
	\frac{X(t_{i+1})-X(t_i)}{\lambda_i} = X'(t_i) + O(\lambda_i).
	\]
	In view of \cref{eq:barxi}, it is easy to obtain $
	\nabla G(\bar x^i) 
	={}\nabla G(x(t_i) ) +O(\lambda_i)$. Hence, recalling	\cref{eq:rescale-NAG-ode}, we have 
	\[
	\left\{
	\begin{aligned}
		{}&z'(t_i)=-  \theta^2(t_i ) \nabla G(x(t_i) ) +O(\lambda_i) , \\
		{}&x'(t_i)=z(t_i)-x(t_i) +O(\lambda_i),\\
		{}&2 \theta (t_i ) \theta'(t_i )= \theta^2(t_i )+O(\lambda_i).
	\end{aligned}
	\right.
	\]
	Note that $t_i$ is independent with $\lambda_i$. Thus, taking $\lambda_i\to0$ leads to the rescaled model \cref{eq:NAG-res-ode}.
\end{appendices}
\begin{appendices}
	\section{Missing Proofs in \cref{sec:res-ode-icpdps}}
	\subsection{Proof of \cref{eq:pf-key-est}}
	\label{app:pf-key-est}
	Since $\nm{K}^2\Theta_0^2=\Phi_0\Psi_0$, we have $\lambda_0=1$, which implies $\Phi_1=\Phi_0+2\gG\Theta_0$ and $\Psi_1 = \Psi_0 + 2\rF\Theta_0$. In view of \cref{eq:para-icpdps}, we obtain
	\begin{equation}\label{eq:Phi-Psi}
		\begin{aligned}
			{}&	\Phi_{i+1}-\Phi_i= 2\gG (\Theta_{i}-\Theta_{i-1})\,\Longrightarrow\, \Phi_{i+1} = \Phi_1+2\gG (\Theta_{i}-\Theta_{0})=\Phi_0+2\gG \Theta_{i},\\
			{}&	
			\Psi_{i+1}-\Psi_{i}= 2\rF(\Theta_{i}-\Theta_{i-1}) \,\Longrightarrow\, \Psi_{i+1} =\Psi_1+2\rF(\Theta_{i}-\Theta_{0}) =\Psi_0+2\rF \Theta_{i}.
		\end{aligned}
	\end{equation}
	Inserting this into \cref{eq:mid} gives 
	\[
	\lambda_{i+1} =\frac{1-\lambda_{i+1}}{S_i},\quad S_i:=  \frac{\nm{K}\Theta_i}{\sqrt{\Phi_0+2\gG \Theta_{i}}\sqrt{\Psi_0+2\rF \Theta_{i}}}.
	\]
	Note that $\Theta_i$ is monotone increasing and so is $S_i$. In view of $\lambda_{i+1} = 1/(1+S_i)$, it is clear that $0<\lambda_{i+1}\leq \lambda_i\leq 1$ for all $i\in\mathbb N$.
	
	It remains to prove \cref{eq:pf-key-est}. Recall that (cf.\cref{eq:Ai})
	\[
	A_i=  1 +\frac{4\gG \rF}{\nm{K}^2}+  \frac{2\gG \Psi_i+2\rF\Phi_i}{\nm{K}\sqrt{\Phi_i\Psi_i}} \geq 1,
	\]
	which means \cref{eq:pf-key-est} is trivial for $i=0$. In the following, assume $i\geq1$ and consider
	\[
	\begin{aligned}
		A_i -1= {}& \frac{4\gG \rF}{\nm{K}^2}+ \frac{2\gG \Psi_i+2\rF\Phi_i}{\nm{K}\sqrt{\Phi_i\Psi_i}} \overset{\text{by \cref{eq:cond-lambdai}}}{=} \frac{4\gG \rF}{\nm{K}^2}+\frac{2\gG \Psi_i+2\rF\Phi_i}{\nm{K}^2\lambda_i\Theta_i}\\
		\overset{\text{by \cref{eq:Phi-Psi}}}{=}{}& \frac{4\gG \rF}{\nm{K}^2}+\frac{8\gG\rF\Theta_{i-1}+4\gG \Psi_0+4\rF\Phi_0}{\nm{K}^2\lambda_i\Theta_i}\\
		\leq {}&\frac{1}{\nm{K}^2}\underbrace{\frac{4\gG \rF\lambda_i+8\gG \rF}{\lambda_i} }_{\mathbb I_1}+\frac{1}{\nm{K}^2}
		\underbrace{	\frac{4\gG \Psi_0+4\rF\Phi_0}{\lambda_i\Theta_i}}_{\mathbb I_2},
	\end{aligned}
	\]
	where in the last step we used the monotone relation $\Theta_{i-1}\leq \Theta_i$.
	
	Let us estimate $\mathbb I_1$ and $\mathbb I_2$ one by one. In view of \cref{eq:mid,eq:Phi-Psi}, it holds that (recall $i\geq 1$)
	\[
	\lambda_{i}
	= 
	\frac{\sqrt{\Phi_{i}\Psi_{i}}}{\nm{K}\Theta_{i-1}}(1-\lambda_{i})\geq \frac{2\sqrt{\gG\rF}}{\nm{K}}(1-\lambda_{i})\quad\Longrightarrow\quad \lambda_{i}\geq \frac{2\sqrt{\gG\rF}}{\nm{K}+2\sqrt{\gG\rF}},
	\]
	which implies
	\begin{equation}\label{eq:est-1}\small
		\begin{aligned}
			\mathbb I_1= \frac{4\gG \rF\lambda_i+8\gG \rF}{\lambda_i} \leq {}&
			(\nm{K}+2\sqrt{\gG\rF})^2\left(\lambda_i^2+2 \lambda_i\right)
			\leq {}	6\lambda_i\left(\nm{K}^2+4\gG\rF\right).
		\end{aligned}
	\end{equation}
	On the other hand, we also have
	\[
	\lambda_{i}= \frac{\sqrt{ \Phi_i\Psi_{i}}}{\nm{K}\Theta_{i}}\geq \frac{\sqrt{ 2(\gG \Psi_0+\rF\Phi_0) \Theta_{i-1}}}{\nm{K}\Theta_{i}}=
	\frac{\sqrt{ 2(\gG \Psi_0+\rF\Phi_0) }}{\nm{K}\sqrt{\Theta_{i-1}}}(1-\lambda_i),
	\]
	which leads to
	\[
	\lambda_i\geq \frac{\sqrt{2(\gG \Psi_0+\rF\Phi_0)}}{\nm{K}\sqrt{\Theta_{i-1}}+\sqrt{2(\gG \Psi_0+\rF\Phi_0)}}
	\geq   \frac{\sqrt{2(\gG \Psi_0+\rF\Phi_0)}}{\nm{K}+\sqrt{2(\gG \Psi_0+\rF\Phi_0)/\Theta_0}}\frac{1}{\sqrt{\Theta_{i}}}.
	\]
	Hence, we obtain 
	\[
	\begin{aligned}
		\mathbb I_2=		  \frac{4\gG \Psi_0+4\rF\Phi_0}{ \lambda_i\Theta_i}
		\leq
		{}& 2\lambda_i   \left(\nm{K}+\sqrt{2(\gG \Psi_0+\rF\Phi_0)/\Theta_0}\right)^2\\
		\leq {}&4\lambda_i \left(\nm{K}^2+\frac{2\gG \Psi_0+2\rF\Phi_0}{\Theta_0}\right).
	\end{aligned}
	\]
	Combining this with \cref{eq:est-1} gives 
	\[
	\begin{aligned}
		A_i-1\leq {}&6\lambda_i\left(1+\frac{4\gG\rF}{\nm{K}^2}\right)+4\lambda_i \left(1+\frac{2\gG \Psi_0+2\rF\Phi_0}{\nm{K}^2\Theta_0}\right)=C_0\lambda_i.
	\end{aligned}
	\]		
	This completes the proof of \cref{eq:pf-key-est}.
	\subsection{Proof of \cref{eq:diff-zeta-eta}}
	\label{app:pf-diff-icpdps}
	Let us start with  \cref{eq:CP-PPA-mod-IC}:
	\[
	\left\{
	\begin{aligned}
		{}&		0\in \tilde{H}_{i+1}(u^{i+1}) + M_{i+1}(u^{i+1}-\bar u^i)+\check{M}_{i+1}(u^{i+1}-u^i),\\
		{}&		\bar{u}^{i+1} = u^{i+1} + \Lambda_{i+2}(\Lambda_{i+1}^{-1}-I)(u^{i+1}-u^i).
	\end{aligned}
	\right.
	\]
	According to \cref{sec:icpdps}, the operator settings are
	\[
	\begin{aligned}
		\tilde{H}_{i+1}(u^{i+1}) = W_{i+1}H(u^{i+1}),\quad 	M_{i+1}  =\begin{pmatrix}
			I & -\mu_{i+1}^{-1}\tau_i K^*\\
			-\lambda_i^{-1}\sigma_{i+1}\omega_i K& I
		\end{pmatrix},\quad \\
		\check{M}_{i+1} = W_{i+1}\Gamma (\Lambda_{i+1}^{-1}-I),\quad 
		W_{i+1} = \diag{\tau_iI,\sigma_{i+1}I},\quad 
		\Lambda_{i+1}  = \diag{ 
			\lambda_iI,\mu_{i+1}I },
	\end{aligned}
	\]
	where $\omega_i = \frac{\lambda_{i}\phi_i\tau_i}{\lambda_{i+1}\phi_{i+1}\tau_{i+1}}$ and 
	\[
	\Gamma = \begin{pmatrix}
		\gG  I&K^*\\ -K& \rF I
	\end{pmatrix},\quad 
	H(u^{i+1}) =  \begin{pmatrix}
		\partial G(x^{i+1}) + K^*y^{i+1}\\\partial F^*(y^{i+1}) - Kx^{i+1}
	\end{pmatrix}.
	\]
	Observe the two identities:
	\[
	\begin{aligned}
		\check{M}_{i+1}(u^{i+1}- u^i) = {}&W_{i+1}\Gamma (\Lambda_{i+1}^{-1}-I)(u^{i+1}-  u^i) \overset{\text{by \cref{eq:zi}}}{=}W_{i+1}\Gamma  (z^{i+1}-  u^{i+1}),
	\end{aligned}
	\]
	and 
	\[
	\begin{aligned}
		{}& u^{i+1}-\bar  u^i = u^{i+1}-u^{i} - \Lambda_{i+1}(\Lambda_{i}^{-1}-I)(u^{i}-u^{i-1})\\
		={}&u^{i+1}-(I-\Lambda_{i+1})u^{i} - \Lambda_{i+1}\left[u^{i}+(\Lambda_{i}^{-1}-I)(u^{i}-u^{i-1})\right]
		={}\Lambda_{i+1}(z^{i+1}-z^i).
	\end{aligned}
	\]
	Reformulate the inclusion part by that
	\begin{equation}\label{eq:re-ic-pdps-u}
		0\in {} H_{i+1}(u^{i+1}) + W_{i+1}^{-1}M_{i+1}\Lambda_{i+1}(z^{i+1}-z^i)+\Gamma(z^{i+1}-u^{i+1}).
	\end{equation}
	A direct calculation yields that
	\[
	\begin{aligned}
		{}&	W^{-1}_{i+1}M_{i+}\Lambda_{i+1}= \diag{\tau_i^{-1}I,\sigma_{i+1}^{-1}I}M_{i+1}\diag{ 
			\lambda_iI,\mu_{i+1}I }\\
		={}&
		\diag{\tau_i^{-1}I,\sigma_{i+1}^{-1}I}\left(I + \begin{pmatrix}
			O & -\mu_{i+1}^{-1}\tau_i K^*\\
			-\lambda_i^{-1}\sigma_{i+1}\omega_i K& O
		\end{pmatrix}\right)\diag{ 
			\lambda_iI,\mu_{i+1}I }\\
		={}&\diag{\lambda_i/\tau_iI,\mu_{i+1}/\sigma_{i+1}I}
		+ \begin{pmatrix}
			O&-K^*\\-\omega_i K&O
		\end{pmatrix}.
	\end{aligned}
	\]
	Recall that $u^{i}=(x^i,y^i)$ and $z^i=(\zeta^i,\eta^i)$. It follows from \cref{eq:re-ic-pdps-u}  that 
	\[
	\begin{aligned}
		{}&\diag{\lambda_i/\tau_iI,\mu_{i+1}/\sigma_{i+1}I}(z^{i+1}-z^i)\\
		{}&\quad + \diag{\gG  I,\rF I}(z^{i+1}-u^{i+1}) \in-	  \begin{pmatrix}
			\partial G(x^{i+1})+	K^*\eta^i\\ \partial F^*(y^{i+1})
			-K\bar \zeta^{i+1}
		\end{pmatrix},
	\end{aligned}
	\]
	with $\bar \zeta^{i+1}=\zeta^{i+1}+\omega_i(\zeta^{i+1}-\zeta^i)$. In component wise, we obtain 
	\[
	\left\{
	\begin{aligned}	
		{}&	\lambda_i(\zeta^{i+1}-\zeta^i)\in {}\tau_i\left[\gG ( x^{i+1}-\zeta^{i+1})-\partial G(x^{i+1})-K^*\eta^i\right],\\
		{}& 	\mu_{i+1} (\eta^{i+1}-\eta^i) \in\sigma_{i+1}\left[\rF( y^{i+1}-\eta^{i+1})-\partial F^*(y^{i+1})+K\bar \zeta^{i+1}\right].
	\end{aligned}
	\right.
	\]
	To match the integrated sequences in \cref{eq:aggreg-icpdps}, we multiply the above two inclusions by $\phi_i$ and $\psi_{i+1}$ respectively and then use \cref{eq:para-icpdps-1} to get
	\[
	\small  
	\left\{
	\begin{aligned}	
		{}& 	\Phi_i   \frac{\zeta^{i+1}-\zeta^i}{\lambda_i} \in\Theta_{i}\left[\gG ( x^{i+1}-\zeta^{i+1})-\partial G(x^{i+1})-K^*\eta^i\right],\\
		{}&	\Psi_{i+1} \frac{\eta^{i+1}-\eta^i}{\lambda_{i+1}} \in \Theta_{i+1}\left[\rF( y^{i+1}-\eta^{i+1})-\partial F^*(y^{i+1})+K(\zeta^{i+1}+\omega_i(\zeta^{i+1}-\zeta^i))\right],
	\end{aligned}
	\right.
	\]
	where 
	\[
	\omega_i = \frac{\lambda_{i}\phi_i\tau_i}{\lambda_{i+1}\phi_{i+1}\tau_{i+1}} = \frac{\lambda_i\Theta_i}{\lambda_{i+1}\Theta_{i+1}} = \frac{\lambda_i(1-\lambda_{i+1})}{\lambda_{i+1}}=\frac{\lambda_i}{\lambda_{i+1}}-\lambda_i.
	\] 
	This proves \cref{eq:diff-zeta-eta}.
\end{appendices}

\bibliographystyle{abbrv}

\end{document}